\documentclass[12pt, a4paper]{article}

\usepackage{pdfsync}
\usepackage{amsmath}\multlinegap=0pt
\usepackage[latin1]{inputenc}
\usepackage{amsthm}
\usepackage{amssymb}
\usepackage[english]{babel}
\usepackage{hyperref}
\usepackage{graphicx}

\usepackage{upgreek}
\usepackage{mathtools}
\usepackage{bbm}
\usepackage{bm}
\usepackage{esint}
\usepackage{xifthen}
\usepackage{xcolor}

\numberwithin{equation}{section}
\theoremstyle{plain}

\newtheorem{thm}{Theorem}[section]
\newtheorem{corollary}[thm]{Corollary}
\newtheorem{proposition}[thm]{Proposition}
\newtheorem{lemma}[thm]{Lemma}
\newtheorem{definition}[thm]{Definition}
\newtheorem{remark}[thm]{Remark}
\theoremstyle{remark}
\newtheorem{example}[thm]{Example}

\newcommand\pref[1]{(\ref{#1})}

\newcommand*{\R}{\mathbb{R}}
\newcommand*{\N}{\mathbb{N}}
\newcommand*{\Z}{\mathbb{Z}}

\renewcommand{\P}{\mathcal{P}}
\newcommand*{\X}{\mathcal{X}}

\DeclareMathOperator{\E}{\mathbb{E}}
\DeclareMathOperator{\Var}{Var}
\DeclareMathOperator{\ind}{\mathbbm{1}}
\DeclareMathOperator*{\argmin}{arg\,min}
\DeclareMathOperator*{\argmax}{arg\,max}
\DeclarePairedDelimiter{\norm}{\lVert}{\rVert}
\DeclarePairedDelimiter{\abs}{\lvert}{\rvert}
\DeclareMathOperator{\bary}{bar}
\DeclareMathOperator{\Bary}{Bar}
\DeclareMathOperator{\Ent}{Ent}
\newcommand\iom{\int_{\Omega}}
\DeclareMathOperator{\supp}{supp}

\DeclareMathOperator{\diam}{diam}
\DeclareMathOperator{\BV}{BV}

\DeclareMathOperator{\id}{id}
\DeclareMathOperator*{\osc}{osc}

\newcommand\loc{{\mathrm{loc}}}
\newcommand\lip{{\mathrm{Lip}}}

\newcommand*{\clos}{\overline}
\newcommand*{\Omb}{\clos{\Omega}}
\newcommand*{\rhob}{\clos{\rho}}
\let \eps\varepsilon

\let \div \relax
\DeclareMathOperator{\div}{div}
\DeclareMathOperator{\tr}{tr}
\newcommand*{\cM}{\mathcal{M}}
\newcommand*{\cof}{\mathrm{cof}}

\newcommand*{\diff}{\mathrm{d}}

\renewcommand{\d}{\,\mathrm{d}}
\newcommand*{\eqset}{\coloneqq} 

\title{Entropic-Wasserstein barycenters: PDE characterization, regularity and CLT}

\author{Guillaume Carlier\thanks{CEREMADE, Universit\'e Paris Dauphine, PSL,
 and INRIA-Paris, MOKAPLAN,
\texttt{carlier@ceremade.dauphine.fr}}
\and
Katharina Eichinger \thanks{CEREMADE, Universit\'e Paris Dauphine, PSL, and INRIA-Paris, MOKAPLAN,
\texttt{eichinger@ceremade.dauphine.fr}}
\and
Alexey Kroshnin \thanks{Universit\'e Claude Bernard, Lyon 1, IITP RAS, and HSE university 
\texttt{kroshnin@math.univ-lyon1.fr}}
}

\begin{document}

\maketitle

\begin{abstract}
In this paper, we investigate properties of entropy-penalized Wasserstein barycenters introduced in \cite{bigotcazpap} as a regularization of Wasserstein barycenters \cite{agueh-carlier}. After characterizing these barycenters in terms of a system of Monge--Amp\`ere equations, we prove some  global moment and Sobolev bounds as well as higher regularity properties. We finally establish a central limit theorem for entropic-Wasserstein barycenters.

\end{abstract}

tions and systems, central limit theorem.

\medskip

\textbf{MS Classification: 49Q25, 35J96, 60B12}.

\section{Introduction}\label{sec-intro}

Wasserstein barycenters are minimizers of a weighted sum of squared quadratic Wasserstein distances to some fixed family of probability measures. As such, they are a particular instance of Fr\'echet means. In recent years, Wasserstein barycenters have become a popular tool to interpolate between probability measures and found various applications in statistics, image processing and machine learning (see \cite{peyretexture}, \cite{Sriva}, \cite{PZ}, \cite{ZemelProcrustes}). There are also  by now various numerical solvers to compute or approximate them (see \cite{Cuturi}, \cite{COO}, \cite{benetal}, \cite{cupe}). Introduced in \cite{agueh-carlier} in the case of finitely many measures over an euclidean space, the Wasserstein barycenter problem has been extended to the Riemannian setting by Kim and Pass \cite{KimPass} and to the case of infinitely many (possibly random) measures by Pass \cite{Pass}, Bigot and Klein \cite{BK} and Le Gouic and Loubes \cite{gouic2015existence} who established a law of large numbers for empirical  Wasserstein barycenters of i.i.d.\ random measures. Having this law of large number in mind, it is natural to look for error estimates and asymptotic normality of the error between population Wasserstein barycenters and their empirical counterpart. However, establishing  a central limit theorem (CLT) for Wasserstein barycenters and, more generally, for Fr\'echet means over a nonnegatively curved metric space seems to be a delicate task (see \cite{ahidar2019convergence}) except in very particular cases  (dimension one or the case of Gaussians, see \cite{acclt}, \cite{kroshnin2019statistical}). The difficulty is not only due to the fact that the problem is infinite-dimensional but also (and in fact more importantly) to the fact that Wasserstein barycenters are related to an obstacle problem for a system of Monge--Amp\`ere equations (see \cite{agueh-carlier}). The support of the Wasserstein barycenter is indeed an unknown of the problem and very little is known about its regularity (see \cite{SanWang} for counter-examples to convexity). The free-boundary aspect of  Wasserstein barycenters actually makes the dependence of the barycenter possibly nonsmooth on the sample and thus prevents one from using a delta method.

\smallskip

Bigot, Cazelles and Papadakis in \cite{bigotcazpap} observed that when one discretizes continuous measures, the corresponding (discrete) barycenters exhibit strong oscillations and proposed to add an  entropic penalization to the Wasserstein variance functional to rule out such discretization artefacts. Such regularizations were also considered in a more general setting by the third author in  \cite{kroshnin2018frechet}. Once one adds an entropic term, the free-boundary aspect of the un-regularized Wasserstein problem disappears and one can expect regularity and quite strong estimates by  PDE arguments. The objective of this paper is precisely to investigate the regularizing effect of the entropic penalty term. Starting from the optimality condition which consists in an elliptic system of Monge--Amp\`ere equations, we will prove various bounds (on the Fisher information, by a maximum principle, or higher regularity based on the regularity theory for Monge--Amp\`ere). We will then consider the stochastic setting of entropic Wasserstein barycenters of random i.i.d.\ measures. As a consequence of our estimates, we will  obtain a strengthened form of the law of large numbers (that is, not only for a.s.\ convergence in Wasserstein distance, but also for Sobolev norms) and more importantly, under suitable additional assumptions, we will obtain a CLT.

\smallskip

The paper is organized as follows. In section~\ref{sec-assump}, we introduce the setting and prove existence and uniqueness of the entropic Wasserstein barycenter. The entropic barycenter is then characterized by a system of Monge--Amp\`ere equations in section~\ref{sec-charact} where we treat the Gaussian case as a simple application. Section~\ref{sec-properties} is devoted to further properties: global moment and Sobolev bounds, strong stability and a maximum principle. Higher regularity is considered in section~\ref{sec-reg} first in the bounded case and then on $\R^d$ for log-concave measures. Section~\ref{sec-CLT} deals with asymptotic results for entropic barycenters of empirical measures with a law of large numbers and a CLT. Finally, the appendix gathers some material related to the linearization of Monge--Amp\`ere equations and to auxiliary probability results which are used in the proof of our CLT. 
\section{Setting, assumptions and preliminaries}\label{sec-assump}

We denote by $\P_2(\R^d)$ the set of Borel probability measures on $\R^d$ having a finite second moment and equip $\P_2(\R^d)$ with the $2$-Wasserstein metric, $W_2$. Recall that for $(\mu, \nu) \in \P_2(\R^d)^2$, the squared $2$-Wasserstein distance between $\mu$ and $\nu$ is defined as the value of the optimal transport problem:
\begin{equation}\label{defW2}
	W_2^2(\mu, \nu) \eqset \inf_{\gamma \in \Pi(\mu, \nu)} \int_{\R^d\times \R^d} \norm{x-y}^2 \d \gamma(x,y)
\end{equation}
where $\Pi(\mu, \nu)$ denotes the set of transport plans between $\mu$ and $\nu$, i.e.\ the set of Borel probability measures on $\R^d\times \R^d$ having $\mu$ and $\nu$ as marginals. Equipped with $W_2$, $\P_2(\R^d)$ is a Polish space, which we shall simply call the Wasserstein space. Convergence in $W_2$ is well known to be equivalent to narrow convergence, plus convergence of the second moment. For more on the optimal transport problem and Wasserstein spaces we refer to the textbooks of Villani \cite{Villani1, Villani2} and Santambrogio \cite{Santambrogio}. The Kantorovich duality formula enables one to express $\frac{1}{2}W_2^2(\mu, \nu)$ as the supremum of 
\[ \int_{\R^d} u \d \mu + \int_{\R^d} v \d \nu\]
among pairs of potentials $u$ and $v$ such that 
\[u(x)+v(y) \leq \frac{1}{2} \norm{x-y}^2, \; \forall (x,y)\in \R^d\times \R^d.\]
In the Kantorovich dual, $u$ and $v$ can be chosen to be conjugate to each other in the sense that
\[u(x) = \inf_{y\in \R^d} \left\{ \frac{1}{2} \norm{x-y}^2 - v(y) \right\}, \; 
v(y) = \inf_{x\in \R^d} \left\{ \frac{1}{2} \norm{x-y}^2 - u(x)\right\}.\]
Hence $u$ and $v$ can be chosen semi-concave and the previous relations can conveniently be expressed in terms of the convex potentials
\begin{equation}\label{defdefi}
	\varphi(x) \eqset \frac{1}{2} \norm{x}^2 - u(x), \; \psi(y) \eqset \frac{1}{2} \norm{y}^2 - v(y)
\end{equation}
by the fact that $\varphi$ and $\psi$ are Legendre transforms of each other
\[\varphi=\psi^*, \; \psi=\varphi^*.\]
The existence of optimal potentials $u$ and $v$ for the Kantorovich dual is well known. An important result due to Brenier \cite{Brenier} says that if, in addition, $\mu$ is absolutely continuous with respect to the Lebesgue measure, then the optimal transport problem in \eqref{defW2} has a unique solution and is given by $\gamma=(\id, \nabla \varphi)_\# \mu$ where $\varphi$ is the convex potential defined in \eqref{defdefi}. The Brenier map $\nabla \varphi$ is unique up to a $\mu$-negligible set, and is the unique gradient of a convex map transporting $\mu$ to $\nu$. If $\mu$ is absolutely continuous and almost everywhere strictly positive, the potential $\varphi$ is unique up to an additive constant. In this case, we denote it $\varphi_\mu^\nu$ and likewise denote by $u_\mu^\nu$ the semi concave Kantorovich potential $u_\mu^\nu = \frac{1}{2} \norm{\cdot}^2 - \varphi_\mu^\nu$.

Now we give ourselves a Borel (with respect to the Wasserstein metric) probability measure $P$ on $\P_2(\R^d)$ such that 
\begin{equation}\label{finitemomentP}
	\int_{\P_2(\R^d)} m_2(\nu) \d P(\nu) < +\infty,
\end{equation}
where $m_2(\nu)$ denotes the second moment of $\nu$ i.e.
\begin{equation}\label{defm2}
	m_2(\nu) = \int_{\R^d} \norm{x}^2 \d \nu(x), \; \forall \nu \in \P_2(\R^d).
\end{equation}
Given a regularization parameter $\lambda>0$ and $\Omega$ a nonempty open connected subset of $\R^d$ with a Lebesgue negligible boundary (of particular interest is the case where $\Omega=\R^d$ or $\Omega$ is convex), we are interested in the following problem (which was introduced in \cite{bigotcazpap} as an entropic regularization of the Wasserstein barycenter problem):
\begin{equation}\label{barycentpbm}
	\inf_{\rho \in \P_2(\R^d)} V_{P, \lambda, \Omega}(\rho) \eqset \frac{1}{2} \int_{\P_2(\R^d)} W_2^2(\rho, \nu) \d P(\nu) + \lambda \Ent_{\Omega}(\rho)
\end{equation}
where $\Ent_{\Omega}$ is defined for every $\mu \in \P_2(\R^d)$ by
\[\Ent_{\Omega} (\mu) = \begin{cases}
    \int_\Omega \rho \log \rho, & \text{ if } \mu = \rho dx \text{ and } \int_{\Omega} \rho = 1,\\
    +\infty & \text{otherwise.}
    \end{cases}
\]
If $\Omega = \R^d$ we simply denote $\Ent_{\R^d} = \Ent$ and $V_{P, \lambda, \R^d} = V_{P, \lambda}$. 

\begin{example}
	If $\Omega = \R^d$ and $P = \sum_{i=1}^I p_i \delta_{\delta_{x_i}}$ is concentrated on Dirac masses, \eqref{barycentpbm} can be reformulated as 
	\[
	\inf_{\rho \in \P_2(\R^d)} \frac{1}{2} \int_{\R^d} \norm[\Big]{x-\sum_{i=1}^I p_i x_i}^2 \d \rho(x) + \lambda \Ent(\rho)
	\]
	whose solution is the Gaussian
	\[
	\rho(x) \eqset \frac{1}{(2 \pi \lambda)^{\frac{d}{2}}} \exp\Bigl(-\frac{1}{2 \lambda} \norm[\Big]{x-\sum_{i=1}^I p_i x_i}^2 \Bigr)
	\]
	whereas the (unregularized) Wasserstein barycenter of $P$ is just $\delta_{\sum_{i=1}^I p_i x_i}$. 
\end{example}

By the direct method of the calculus of variations, one easily obtains

\begin{proposition}\label{existbar}
	Assuming \eqref{finitemomentP}, then \eqref{barycentpbm} admits a unique solution.
\end{proposition}

\begin{proof}
	Let $(\rho_n)$ be a minimizing sequence for \eqref{barycentpbm} (necessarily $0$ outside $\Omb$). In what follows $C$ will denote a constant which may vary from one line to the other. Observing that 
	\[ \frac{1}{2} W_2^2(\rho_n, \nu) \geq \frac{1}{4} m_2(\rho_n) - m_2(\nu),\]
	we deduce from the fact that $ V_{P, \lambda, \Omega}(\rho_n)$ is bounded from above and \eqref{finitemomentP} that we have 
	\begin{equation}\label{bdm}
		\frac{1}{4} m_2(\rho_n) + \lambda \Ent_{\Omega}(\rho_n) \leq C.
	\end{equation}
	It now follows from \cite{jko} that for $\alpha \in (\frac{d}{d+2}, 1)$, one can bound from below the entropy by 
	\begin{equation}\label{bdentm}
		\Ent_{\Omega}(\rho) \geq - C(1+m_2(\rho))^{\alpha}
	\end{equation}
	with \eqref{bdm} this shows that $m_2(\rho_n)$ is bounded so that $(\rho_n)$ is tight. One may therefore assume, taking a subsequence if necessary that $(\rho_n)$ converges narrowly to some $\rhob$. Of course $\rhob\in \P_2(\R^d)$ (with $m_2(\rhob) \leq \liminf_n m_2(\rho_n)$) and $\rhob$ vanishes outside $\Omb$. Now since $m_2(\rho_n)$ is bounded and $(\rho_n)$ converges narrowly to $\rhob$ we have (see e.g.\ the appendix of \cite{cdps} for details):
	\begin{equation}\label{liminfent}
		\Ent_{\Omega}(\rhob)\leq \liminf_n \Ent_{\Omega}(\rho_n).
	\end{equation}
	and $\Ent_{\Omega}(\rhob)>-\infty$ thanks to \eqref{bdentm}.
	We also have for every $\nu\in \P_2(\R^d)$
	\[W^2_2(\rhob, \nu) \leq \liminf_n W_2^2(\rho_n, \nu)\]
	hence, by Fatou's Lemma:
	\[\int_{\P_2(\R^d)} W_2^2(\rhob, \nu) \d P(\nu) \leq  \liminf_n \int_{\P_2(\R^d)} W_2^2(\rho_n, \nu) \d P(\nu)\]
	which, together with \eqref{liminfent} enables us to conclude that $\rhob$ solves \eqref{barycentpbm}. The uniqueness of the minimizer directly follows from the strict convexity of the entropy and the convexity of the Wasserstein term. 
\end{proof}

Entropic-Wasserstein barycenters can therefore be defined as follows:

\begin{definition}
	The unique solution $\rhob$ of \eqref{barycentpbm} is called the entropic-Wasserstein barycenter of $P$ with respect to $\lambda$ and $\Omega$ and denoted $\bary_{\lambda, \Omega}(P)$ and simply $\bary_{\lambda}(P)$ if $\Omega = \R^d$.
\end{definition}
\begin{remark} Theorem~4.4 in \cite{bigotcazpap} states under the additional assumption that $\Omega$ is convex and compact (note that taking the closure of $\Omega$ does not change the entropic-Wasserstein barycenter) and $\supp \nu \subset \Omega$ for $P$-a.e.\ $\nu$, then 
	\[
	W_2(\bary_{\lambda, \Omega}(P), \bary_{0, \Omega}(P)) \rightarrow 0 \quad\text{as}\quad \lambda \rightarrow 0,
	\]
	provided $\bary_{0, \Omega}(P)$ is unique. By inspecting their proof, one can actually see that the compactness assumption on $\Omega$ can be relaxed by the same argument as in Proposition~\ref{existbar}. The assumption on the inclusion of the support can also be omitted if one understands $\bary_{0, \Omega}(P)$ as the Wasserstein barycenter of $P$ constrained to have support in $\Omega.$
\end{remark}
 
Since $\bary_{\lambda, \Omega}(P)$ is absolutely continuous with respect to the Lebesgue measure, we shall slightly abuse notations and use the same notation for its density.

We can immediately state some basic invariance properties of entropic-Wasserstein barycenters in case $\Omega = \R^d$. For instance, if we shift all measures $\nu$ by some vector $s \in \R^d$ and rotate by some orthogonal matrix $Q \in O(d)$, then entropic-Wasserstein barycenters will be also shifted and rotated by the same vector and matrix (clearly, the same result holds for any subgroup of translations and orthogonal transformations that $\Omega$ is invariant to). The next proposition shows that translations can actually  be ``factored out'' from the barycenter.

\begin{proposition}\label{prop:translation}
	Let $\Omega = \R^d$, $\lambda > 0$, $P$ be a measure on $\P_2(\R^d)$ satisfying condition \eqref{finitemomentP}, and $\rhob = \bary_{\lambda}(P)$. Fix a measurable map $\bm{s} \in L^2\left(P; \R^d\right)$ and define a measure
	$P_{\bm{s}} \eqset (\nu \mapsto \nu + \bm{s}(\nu))_\# P$, where $\nu \oplus s \eqset (x \mapsto s + x)_\# \nu$ for all $\nu \in \P_2(\R^d)$ and $s \in \R^d$.
	Then $\bary_{\lambda}(P_{\bm{s}}) = \rhob \oplus \bar{s}$, with $\bar{s} \eqset \int_{\P_2(\R^d)} \bm{s}(\nu) \d P(\nu)$.
\end{proposition}

\begin{proof}
	Note that it is enough to consider the case $\E_\nu[X] = 0$ for $P$-a.e.\ $\nu$, where $\E_\nu[X] = \int_{\R^d} x \d \nu(x)$ is the average of $\nu \in \P_2(\R^d)$. 
	Recall that due to the bias-variance decomposition
	\[
	W_2^2(\mu, \nu) = W_2^2(\mu \ominus \E_\mu[X], \nu \ominus \E_\nu[X]) + \norm{\E_\mu[X] - \E_\nu[X]}^2, \quad \mu, \nu \in \P_2(\R^d).
	\]
	Since entropy is invariant to shifts, we get for any $\rho \in \P_2(\R^d)$ and $a \in \R^d$
	\begin{align*}
		V_{P_{\bm{s}}, \lambda}(\rho \oplus a) &= \frac{1}{2} \int_{\P_2(\R^d)} \bigl[W_2^2(\rho \ominus \E_{\rho}[X], \nu) + \norm*{\E_\rho[X] + a - \bm{s}(\nu)}^2\bigr] \d P(\nu)  \\
		&\quad\quad + \lambda \Ent(\rho) \\
		&= V_{P, \lambda}(\rho) -  \frac{1}{2} \norm{\E_\rho[X]}^2 +   \frac{1}{2} \norm*{a + \E_\rho[X] - \bar{s}}^2 + C.
	\end{align*}
	In particular, taking $\bm{s} \equiv 0$, $\rho = \rhob$, and using that the minimum with respect to $a$ is attained at $0$, we get that $\E_{\rhob}[X] = 0$. Now, we can first minimize $V_{P, \lambda}(\rho)$ over $\rho$'s with zero mean: $\E_\rho[X] = 0$, and then minimize the third term with respect to $a$, hence $\bary_\lambda(P_{\bm{s}}) = \rhob \oplus a$, $a = \bar{s}$.
	The claim follows.
\end{proof}

\begin{remark}
	Note that, when $\Omega=\R^d$, a useful corollary of Proposition~\ref{prop:translation} is that the average of entropic-Wasserstein barycenter is the expectation of averages:
	\begin{equation}\label{eq:average}
		\E_{\rhob}[X] = \int_{\P_2(\R^d)} \E_\nu[X] \d P(\nu).
	\end{equation}
\end{remark}

\section{Characterization}\label{sec-charact}

The entropic term forces the regularized barycenter to be everywhere positive. Indeed, arguing in a similar way as in Lemma~8.6 from \cite{Santambrogio}, we arrive at:
\begin{lemma}\label{rhobpos}
	Let $\rhob \eqset \bary_{\lambda, \Omega}(P)$ then $\rhob>0$ a.e.\ on $\Omega$ and $\log(\rhob) \in L^1_\loc(\Omega)$.
\end{lemma}

\begin{proof}
	Let $g$ be a Gaussian density, scaled so as to give mass $1$ to $\Omega$. For $t \in (0,1)$, set $\rho_t \eqset (1-t) \rhob + t g$. The convexity of $\rho \mapsto W_2^2(\rho, \nu)$ together with the optimality of $\rhob$, give
	\[
	\lambda (\Ent_{\Omega}(\rho_t) - \Ent_{\Omega}(\rhob)) \geq \frac{t}{2} \int_{\P_2(\R^d)} [W_2^2(\rhob, \nu) - W_2^2(g, \nu) ] \d P(\nu)
	\]
	so that for some $C$, we have for every $t \in (0,1)$,
	\begin{equation}\label{lbentrt}
		\frac{1}{t} (\Ent_{\Omega}(\rho_t) - \Ent_{\Omega}(\rhob)) \geq C.
	\end{equation}
	Now, observe that 
	\begin{align*}
		\frac{1}{t} (\Ent_{\Omega}(\rho_t)-\Ent_{\Omega}(\rhob)) &= \int_{\{\rhob=0\}} g \log(tg) + \int_{\{\rhob>0\}} \frac{1}{t} (\rho_t \log(\rho_t) - \rhob \log(\rhob)) \\
		& \leq \int_{\{\rhob=0\}} g \log(t g) + \int_{\{\rhob>0\}} (g \log(g) - \rhob \log(\rhob))\\
		&\leq \log(t) \int_{\{\rhob=0\}} g + \Ent_{\Omega}(g) - \Ent_{\Omega}(\rhob)
	\end{align*}
	(where in the second line we have used the convexity of $s \mapsto s \log(s)$). Combining this inequality with \eqref{lbentrt} and letting $t \to 0^+$, we immediately see that $|\{\rhob=0\}| = 0$. 
	
	\smallskip
	
	Let us now show that $\log(\rhob) \in L^1_\loc(\Omega)$. Since $\max(0, \log(\rhob)) \leq \rhob$ we have to show that $\int_K \log(\rhob) > -\infty$ for every compact subset (of positive Lebesgue measure) $K$ of $\Omega$. Calling $\mu$ the uniform probability measure on $K$, setting $\nu_t \eqset \rhob + t (\mu - \rhob)$ for $t \in (0,1)$ and arguing as above, we have 
	\[
	\frac{1}{t} (\Ent_{\Omega}(\nu_t) - \Ent_{\Omega}(\rhob)) \geq C,
	\]
	moreover $\frac{1}{t} (\nu_t \log(\nu_t)-\rhob \log \rhob) \leq \mu \log(\mu) - \rhob \log \rhob \in L^1(\Omega)$, Fatou's Lemma and the previous inequality thus give
	\begin{align*}
		C &\leq \limsup_{t \to 0^+} \frac{1}{t} (\Ent_{\Omega}(\nu_t) - \Ent_{\Omega}(\rhob))\\
		&\leq \int_{\Omega} \limsup_{t \to 0^+} (\nu_t \log(\nu_t) - \rhob \log (\rhob))
		= \int_{\Omega} \log(\rhob)(\mu-\rhob)
	\end{align*}
	and since $\Ent_{\Omega}(\rhob)$ is finite, this gives $\int_K \log(\rhob) > -\infty$.
\end{proof}

The fact that the regularized barycenter is everywhere positive guarantees uniqueness (up to a constant) of the Kantorovich potential between $\rhob$ and $\nu \in \P_2(\R^d)$. This uniqueness is well-known to be very useful in terms of differentiability of $\mu \mapsto W_2^2(\mu, \nu)$ at $\mu = \rhob$ as expressed in Lemma~\ref{diffw2} below (which is slight generalization of Proposition~7.17 in \cite{Santambrogio}). The following inequality will be useful to justify the differentiability of $\mu \mapsto \int_{\P_2(\R^d)} W_2^2(\mu, \nu) \d P(\nu)$ at $\mu = \rhob$:

\begin{lemma}\label{harnacklem}
	Let $\rho \in L^1(\Omega)$ and $\rho > 0$ a.e.\ on $\Omega$. Then for any compact set $K \subset \Omega$ and any convex function $\varphi \colon \Omega \to \R$ one has
	\begin{equation}\label{harnackconv}
		\osc_K \varphi 
		\eqset \max_K \varphi - \min_K \varphi 
		\le \frac{\diam(K) + r}{\inf\limits_{x \in K_{r/2}} \rho(B_{r/2}(x))} \iom \norm{\nabla \varphi} \rho,
	\end{equation}
	where $0 < r \leq d(K, \partial \Omega)$ and $K_{\sigma} = \bigcup_{x \in K} \bar{B}_{\sigma}(x)$ for any $\sigma > 0$.
	Moreover, the Lipschitz constant of $\varphi$ on $K$, $\lip\left(\varphi|_K\right)$, can be estimated as
	\begin{equation}\label{eq:lip_bound_conv}
		\lip\left(\varphi|_K\right) 
		\le \frac{2 \diam(K) + 3 r}{r \inf\limits_{x \in K_{3 r / 4}} \rho(B_{r / 4}(x))} \iom \norm{\nabla \varphi} \rho.
	\end{equation}
\end{lemma}

\begin{remark}
	Notice that $\Omega$ is not necessary convex, thus we say a function $\varphi$ on $\Omega$ is convex if it can be extended to a convex function on $\R^d$ (possibly taking value $+\infty$), see \cite{figalli2017monge}.
\end{remark}

\begin{proof}
	Let $x_1 \in \argmax_{K} \varphi$, $x_0 \in \argmin_{K} \varphi$, and $w \in \partial \varphi(x_1)$. Then for any $x \in \Omega$ and $z \in \partial \varphi(x)$ one has
	\[
	\varphi(x_0) + z \cdot (x - x_0) \ge \varphi(x) \ge \varphi(x_1) + w \cdot (x - x_1),
	\]
	and thus the Cauchy--Schwarz inequality yields
	\[
	\norm{z} \ge \frac{\osc_K \varphi + w \cdot (x - x_1)}{\norm{x - x_0}}.
	\]
	Since $\varphi$ is a.e.\ differentiable, we have
	\begin{align*}
		\iom \norm{\nabla \varphi} \rho
		&\ge \int_{W_r(x_1,w)} \norm{\nabla \varphi} \rho
		\ge \osc_K \varphi \int_{W_r(x_1,w)} \frac{1}{\norm{x - x_0}} \rho(x) \d x \\
		& \ge \osc_K \varphi \int_{W_r(x_1,w)} \frac{1}{\norm{x - x_1} + \norm{x_1 - x_0}} \rho(x) \d x\\
		& \ge \osc_K \varphi \, \frac{\rho\left(B_{r/2}\left(x + \tfrac{r w}{2 \norm{w}}\right)\right)}{\diam(K) + r},
	\end{align*}
	where we have set $W_r(x, w) \eqset \left\{y \in B_r(x) : w \cdot (y - x) \ge 0\right\}$ and used the fact that $B_{r/2}\left(x + \tfrac{r w}{2 \norm{w}}\right) \subset W_r(x, w)$ and $x + \tfrac{r w}{2 \norm{w}} \in K_{r/2}$.
	Finally, the positivity of $\rho$ together with the compactness of $K$ implies that 
	\[
	\inf\left\{\rho(W_r(x, w)) : x \in K,\; w \in \R^d\right\} 
	\ge \inf_{x \in K_{r/2}} \rho(B_{r/2}(x)) > 0.
	\]
	The first claim follows.
	
	To prove~\eqref{eq:lip_bound_conv} we apply~\eqref{harnackconv} to $K_{r / 2}$, which yields
	\begin{align*}
		\osc_{K_{r / 2}} \varphi
		&\le \frac{\diam(K_{r / 2}) + r / 2}{\inf\limits_{x \in K_{3 r / 4}} \rho\left(B_{r / 4}(x)\right)} \iom \norm{\nabla \varphi} \rho \\
		&\le \frac{\diam(K) + 3 r/2}{\inf\limits_{x \in K_{3 r / 4}} \rho\left(B_{r / 4}(x)\right)} \iom \norm{\nabla \varphi} \rho.
	\end{align*}
	Note that for any $x \in K$ and $w \in \partial \varphi(x)$ one has $B_{r / 2}(x) \subset K_{r / 2}$, hence
	\[
	\osc_{K_{r / 2}} \varphi \ge \osc_{B_{r / 2}(x)} \varphi \ge \frac{r}{2} \norm{w}.
	\]
	Therefore,
	\[
	\norm{w} \le \frac{2}{r} \osc_{K_{r / 2}} \varphi 
	\le \frac{2 \diam(K) + 3 r}{r \inf\limits_{x \in K_{3 r / 4}} \rho\left(B_{r / 4}(x)\right)} \iom \norm{\nabla \varphi} \rho,
	\]
	thus we obtain the desired bound on $\lip\left(\varphi|_K\right)$.
\end{proof}

\begin{lemma}\label{diffw2}
	Let $\rhob \eqset \bary_{\lambda, \Omega}(P)$ and given $\nu \in \P_2(\R^d)$, let $u_{\rhob}^\nu$ be the (unique on $\Omega$, up to an additive constant) Kantorovich potential between $\rhob$ and $\nu$. Let $\mu \in L^1(\Omega)$ be a probability density such that $\mu - \rhob$ has compact support in $\Omega$, defining $\rho_\eps \eqset \rhob + \eps (\mu - \rhob)$ for $\eps \in (0, \frac{1}{2})$, we have
	\[
	\lim_{\eps \to 0^+} \frac{1}{2 \eps} [W_2^2(\rho_\eps, \nu) - W_2^2(\rhob, \nu)] 
	= \int_{\Omega} u_{\rhob}^\nu \d (\mu - \rhob)
	\]
	and
	\[
	\lim_{\eps \to 0^+} \frac{1}{2 \eps} \int_{\P_2(\R^d)} [W_2^2(\rho_\eps, \nu) - W_2^2(\rhob, \nu)] \d P(\nu) 
	= \int_{\P_2(\R^d)} \Bigl( \int_{\Omega} u_{\rhob}^\nu \d (\mu - \rhob)\Bigr) \d P(\nu).
	\]
\end{lemma}

\begin{proof}
	Let us shorten notations by defining 
	\[
	u \eqset u_{\rhob}^{\nu}, \; \varphi \eqset \varphi_{\rhob}^{\nu} = \frac{1}{2} \norm{\cdot}^2 - u_{\rhob}^{\nu}
	\]
	and let $u_\eps$ be a Kantorovich potential between $\rho_\eps$ and $\nu$ and $\varphi_\eps \eqset \frac{1}{2} \norm{\cdot}^2 - u_\eps$. Let $K$ be a compact subset of $\Omega$ supporting $\mu - \rhob$ and normalize the potentials $\varphi$ and $\varphi_\eps$ in such a way that their minimum on $K$ is $0$. It immediately follows from the Kantorovich duality formula that
	\begin{equation}\label{ineqdw2}
		\int_{K} u_\eps \d (\mu - \rhob) \geq \frac{1}{2 \eps} [W_2^2(\rho_\eps, \nu) - W_2^2(\rhob, \nu)] 
		\geq \int_{K} u \d (\mu - \rhob).
	\end{equation} 
	Now observe that since $(\nabla \varphi_{\eps})_\# \rho_\eps = \nu$, we have
	\begin{equation}\label{boundwm2}
		m_2(\nu) = \int_{\Omega} \norm{\nabla \varphi_\eps}^2 \rho_\eps 
		\geq (1 - \eps) \int_{\Omega} \norm{\nabla \varphi_\eps}^2 \rhob.
	\end{equation}
	We then deduce from Jensen's inequality a bound on $\int_{\Omega} \norm{\nabla \varphi_\eps} \rhob$ which does not depend on $\eps$. Thanks to Lemma~\ref{harnacklem}, we obtain local uniform bounds on $\varphi_\eps$ and therefore can deduce that for some vanishing sequence of $\eps_n$, $\varphi_{\eps_n}$ converges locally uniformly on $\Omega$ to some convex $\psi$ whose minimum is $0$ on $K$ and that $\nabla \varphi_{\eps_n}$ converges $\rhob$-a.e.\ to $\nabla \psi$. Using continuous bounded test-functions and Lebesgue's dominated convergence theorem, we can pass to the limit in ${\nabla \varphi_{\eps_n}}_\# \rho_{\eps_n} = \nu$ to deduce that $\nabla \psi_\# \rhob = \nu$, which, with our normalization and the uniqueness of Brenier's map implies that $\psi = \varphi$ and also full uniform convergence on $K$ of $u_\eps$ to $u$. Passing to the limit in \eqref{ineqdw2} gives the first claim of the Lemma. 
	
	To prove the second claim, set
	\[
	\theta_\eps(\nu) \eqset \frac{1}{2 \eps} [W_2^2(\rho_\eps, \nu) - W_2^2(\rhob, \nu)]
	\]
	and observe that it follows from \eqref{ineqdw2}-\eqref{boundwm2} and Lemma~\ref{harnacklem} that $\theta_\eps(\nu)$ can be bounded from above and from below by two affine functions of $m_2(\nu)$, the desired result therefore follows from \eqref{finitemomentP}, Lebesgue's dominated convergence theorem and the first claim. 
\end{proof}

We are now in position to characterize the regularized barycenter:

\begin{proposition}\label{prop:characterization}
	Let $\rhob \in \mathcal{P}_2(\R^d)$, then $\rhob = \bary_{\lambda, \Omega}(P)$ if and only if, denoting by $\nabla \varphi_{\rhob}^{\nu}$ Brenier's map between $\rhob$ and $\nu$, there are normalizing constants for $\varphi_{\rhob}^\nu$ such that
	$\rhob$ has a continuous density given by
	\begin{equation}\label{densrhob}
		\rhob(x) \eqset \exp\Bigl(-\frac{1}{2 \lambda} \norm{x}^2 + \frac{1}{\lambda} \int_{\P_2(\R^d)} \varphi_{\rhob}^{\nu}(x) \d P(\nu) \Bigr)
	\end{equation}
	for every $x\in \Omega$. Moreover, $\log(\rhob)$ is semi-convex hence differentiable a.e.\ and for a.e.\ $x\in \Omega$, one has
	\begin{equation}\label{gradopticond}
		x + \lambda \nabla \log(\rhob)(x) = \int_{\P_2(\R^d)} \nabla \varphi_{\rhob}^{\nu}(x) \d P(\nu).
	\end{equation}

\end{proposition}
 
\begin{proof}
	For necessity fix a compact with nonempty interior subset $K$ of $\Omega$ and normalize $u_{\rhob}^\nu$ such that it has minimum $0$ on $K$, then, arguing as in the proof of Lemma~\ref{diffw2}, there is a constant $C_K$ such that $\norm*{u_{\rhob}^\nu}_{L^\infty}(K) \leq C_K(1+m_2(\nu))$ so that the (semi-concave) potential 
	\[
	x \mapsto U(x) \eqset \int_{\P_2(\R^d)} u_{\rhob}^{\nu}(x) \d P(\nu)
	\]
	is bounded on $K$. Now we claim that $V\eqset\lambda \log(\rhob)+ U$ (which is integrable on $K$ thanks to Lemma~\ref{rhobpos}) coincides Lebesgue a.e.\ with a constant on $K$ (which taking an exhaustive sequence of compact subsets of $\Omega$ will enable to find normalizing constants for $\varphi_{\rhob}^{\nu}$ that do not depend on $K$ and therefore prove \eqref{densrhob}). Assume, by contradiction, that $V$ does not coincide Lebesgue a.e.\ with a constant on $K$, then we could find two measurable subsets $K_1$ and $K_2$ of $K$, both of positive Lebesgue measure and $\alpha\in \R$ and $\delta>0$ such that 
	\begin{equation}\label{sepV}
		V \geq \alpha + \delta \text{ a.e.\ on $K_1$}, \; V\leq \alpha -\delta \text{ a.e.\ on $K_2$}.
	\end{equation}
	In particular $\rhob(K_1)>0$ and $\rhob(K_2)>0$, now set $\beta\eqset \frac{\rhob(K_1)}{2 \rhob(K_2)}$ and define the probability density $\mu\in L^1(\Omega)$ by
	\[
	\mu(x) \eqset \begin{cases} \frac{1}{2} \rhob(x) & \text{if } x \in K_1,\\ 
	(1+\beta) \rhob (x) & \text{if } x \in K_2, \\ 
	\rhob(x) & \text{otherwise}\end{cases}
	\]
	and $\rho_\eps \eqset \rhob + \eps (\mu-\rhob)$. It is straightforward to check that
	\[
	\lim_{\eps \to 0^+} \frac{1}{\eps} (\Ent(\rho_\eps)-\Ent(\rhob)) = \int_{K} \log(\rhob)(\mu-\rhob).
	\]
	With Lemma~\ref{diffw2}, the construction of $\mu$ and \eqref{sepV}, this yields
	\begin{multline*}
		\lim_{\eps\to 0^+} \frac{1}{\eps} [V_{P, \lambda, \Omega}(\rho_\eps)- V_{P, \lambda, \Omega}(\rhob)] = \int_{K} V (\mu-\rhob) \\
		= -\frac{1}{2}\int_{K_1} V \rhob + \beta \int_{K_2} V \rhob \leq -\delta \rhob(K_1)<0
	\end{multline*}
	contradicting the fact by optimality of $\rhob$, $V_{P, \lambda, \Omega}(\rho_\eps)\geq V_{P, \lambda, \Omega}(\rhob)$. 

	Now assume that $\rhob \in \P_2(\R^d)$ satisfies \eqref{densrhob}, and let $\mu \in \P_2(\R^d)$ be supported on $\Omb$ and such that $\Ent_{\Omega}(\mu) < \infty$. Using the convexity of the entropy firstly gives
	\begin{equation}\label{enebotan}
		\lambda \Ent_{\Omega}(\mu) \geq \lambda \Ent_{\Omega}(\rhob)+ \lambda \int_{\Omega} \log(\rhob) (\mu-\rhob).
	\end{equation}	 
	Secondly, by Kantorovich duality formula and using the fact that $u_{\rhob}^\nu$ is a Kantorovich potential between $\rhob$ and $\nu$, we get
	\[\begin{split}
	 \frac{1}{2} \int_{\P_2(\R^d)} W_2^2(\mu, \nu) \d P(\nu) \geq \frac{1}{2} \int_{\P_2(\R^d)} W_2^2(\rhob, \nu) \d P(\nu)\\
	+ \int_{\P_2(\R^d)} \Bigl( \int_{\Omega} u_{\rhob}^\nu \d (\mu - \rhob)\Bigr) \d P(\nu).
	\end{split}\]
	Adding \eqref{enebotan}, observing that \eqref{densrhob} means that $\lambda \log \rhob + \int_{\P_2(\R^d)} u_{\rhob}^\nu \d P(\nu)=0$ and using Fubini's theorem, we thus get
	\[V_{P, \lambda, \Omega}(\mu) \geq V_{P, \lambda, \Omega}(\rhob), \]
	so that $\rhob=\bary_{\lambda, \Omega}(P)$.

	\smallskip
	
	Let us now prove \pref{gradopticond}. Since 
	\[
	\Phi \eqset \int_{\P_2(\R^d)} \varphi_{\rhob}^{\nu} \d P(\nu)
	\]
	is convex, $\log \rhob$ is semi-convex. It is therefore differentiable a.e. Now we claim that if $x \in \Omega$ is a differentiability point of $\Phi$ it also has to be a differentiability point of $\varphi_{\rhob}^{\nu}$ for $P$-almost every $\nu$. Indeed, assume that $\Phi$ is differentiable at $x \in \Omega$. For $n \in \N^*$, let $A_n$ denote the set of $\nu \in \P_2(\R^d)$ for which there exist $p_\nu$ and $q_\nu$ in $\partial \varphi_{\rhob}^{\nu}(x)$ such that $\norm{p_\nu - q_\nu} \geq 1/n$. The desired claim will be established if we prove that $P(A_n) = 0$ for every $n \in \N^*$. Let then $(q_\nu, p_\nu) \in \partial \varphi_{\rhob}^{\nu}(x)^2$ be chosen (in a measurable way) so that $\norm{p_\nu - q_\nu} \geq 1/n$ when $\nu \in A_n$, then, for every $h \in \Omega - x$, one has 
	\[
	\varphi_{\rhob}^{\nu}(x+h)- \varphi_{\rhob}^{\nu}(x) - \frac{1}{2} (p_\nu + q_\nu) \cdot h \geq \frac{1}{2} \abs{(p_\nu - q_\nu) \cdot h},
	\]
	so that, by integration $s \eqset \int_{\P_2(\R^d)} \frac{p_\nu + q_\nu}{2} \d P(\nu) \in \partial \Phi(x) = \{\nabla \Phi(x)\}$ and then
	\[
	\Phi(x+h) - \Phi(x) - s \cdot h = o(h) \geq \frac{1}{2} \int_{A_n} \abs{(p_\nu - q_\nu) \cdot h} \d P(\nu).
	\]
	By homogeneity, we thus have $\int_{A_n} \abs{(p_\nu - q_\nu) \cdot h} \d P(\nu) = 0$ for every $h$ so that $\int_{A_n} \norm{p_\nu - q_\nu} \d P(\nu) = 0 \geq P(A_n)/n$ and therefore $P(A_n) = 0$. Hence, if $\Phi$ is differentiable at $x$, for every $h \in \R^d$, we have:
	\[
	\frac{\varphi_{\rhob}^{\nu}(x + t h) - \varphi_{\rhob}^{\nu}(x)}{t}
	\to \nabla \varphi_{\rhob}^{\nu}(x) \cdot h \text{ as } t \to 0^+, \text{ for $P$-a.e.\ $\nu$}.
	\]
	Moreover, the left-hand side above is controlled in absolute value by the Lipschitz constant of $\varphi_{\rhob}^{\nu}$ in a compact neighbourhood of $x$ which, thanks to Lemma~\ref{harnacklem}, in turn, is controlled by 
	\[
	\int_{\Omega} \norm*{\nabla \varphi_{\rhob}^{\nu}} \rhob = \int_{\R^d} \norm{y} \d \nu(y) \leq \sqrt{m_2(\nu)}.
	\]
	Thanks to \eqref{finitemomentP} and Lebesgue's dominated convergence theorem, we thus get
	\[
	\nabla \Phi(x) = \int_{\P_2(\R^d)} \nabla \varphi_{\rhob}^{\nu}(x) \d P(\nu),
	\]
	which shows \eqref{gradopticond}.
 
\end{proof}

\begin{remark}[A first regularizing effect]\label{rem:BVreg}
	One immediately deduces from \eqref{densrhob} and the convexity of $\varphi_{\rhob}^{\nu}$, further regularity properties of the regularized barycenter:
	\begin{equation}\label{BVnablalogro}
		\log(\rhob) \in L^{\infty}_{\loc}(\Omega), \; \rhob \in W^{1, \infty}_{\loc}(\Omega), \text{ and } \nabla \rhob \in \BV_{\loc}(\Omega, \R^d).
	\end{equation}
\end{remark}

\begin{example}[Gaussian case]
	Suppose now that $P$ is concentrated on Gaussian measures and $\Omega = \R^d$; then the regularized barycenter is Gaussian as well. In order to prove this we can assume thanks to Proposition~\ref{prop:translation} that $P$-a.e.\ $\nu = \mathcal{N}(0, S_\nu)$, where $S_\nu$ are positive semi-definite matrices with $\E_P\left[S_\nu\right] \leq \sigma^2 I$, $\sigma > 0$. We want to prove that there is a positive definite symmetric matrix $\bar{S}$ such that
	\[
	\bary_{\lambda}(P) = \mathcal{N}(0, \bar{S}).
	\]
	In order to see that, recall that the optimal transport $T_{\rho}^{\nu}$ from $\rho = \mathcal{N}(0, S)$ to $\nu = \mathcal{N}(0, S_\nu)$ is given by (see e.g.\ \cite{dowson1982frechet})
	\[
	T_{\rho}^{\nu}(x) \eqset \underbrace{S^{-1/2} \left(S^{1/2} S_\nu S^{1/2}\right)^{1/2} S^{-1/2}}_{\eqqcolon T_S^{S_\nu}} x.
	\]
	Thus $\varphi_\rho^\nu = \frac{1}{2} x \cdot T_S^{S_\nu} x + C$, and the optimality condition \eqref{densrhob} can be rewritten as
	\[
	- \frac{\lambda}{2} x \cdot \bar{S}^{-1} x = - \frac{\norm{x}^2}{2} + \frac{1}{2} \E_P \left[x \cdot T_S^{S_\nu} x\right] + C,
	\]
	i.e.
	\[
	I = \lambda \bar{S}^{-1} + \bar{S}^{-1/2} \E_P\left[\left(\bar{S}^{1/2} S_\nu \bar{S}^{1/2}\right)^{1/2}\right] \bar{S}^{-1/2}.
	\]
	Thus $\bar{S}$ has to be a solution of the following fixed-point equation
	\[
	S = \Phi(S) \eqset \lambda I + \E_P\left[\left(S^{1/2} S_\nu S^{1/2}\right)^{1/2}\right].
	\]
	This has a solution by Brouwer's fixed-point theorem. Indeed, denote by $\alpha_\nu$ the largest eigenvalue of $S_\nu$. Then, by assumption
	\[
	\E_P[\alpha_\nu] \leq \tr \E_P[S_\nu] \leq d \sigma^2.
	\]
	Define
	\[
	\alpha \eqset 2 \lambda + d \sigma^2,
	\]
	then for any $\lambda I \le S \le \alpha I$ it holds that
	\begin{align*}
		\Phi(S) &\leq \left(\lambda + \E_P\left[(\alpha_\nu \alpha)^{1/2}\right]\right) I
		\leq \left(\lambda + \frac{\alpha}{2} + \frac{\E_P[\alpha_\nu]}{2}\right) I \\
		& \leq \left(\lambda + \frac{\alpha}{2} + \frac{d \sigma^2}{2}\right) I
		= \alpha I.
	\end{align*}
	So, $\Phi(\cdot)$ maps the convex set $\{\lambda I \le S \le \alpha I\}$ to itself, and it is clearly continuous. The existence of $\bar{S}$ such that $\bar{S} = \Phi(\bar{S})$ therefore follows from Brouwer's fixed-point theorem.
\end{example}

\begin{example}[Discrete case]\label{discreteex}
	Consider now the case where $\Omega=\R^d$ and $P$ is a discrete measure supported on discrete measures:
	\[
	P = \sum_{i\in I} p_i \delta_{\nu_i} \text{ with } \nu_i \eqset \sum_{j\in J_i} \nu_i^j \delta_{x_i^j} \,,
	\]
	where $I$ and each $J_i$ are finite and for every $i\in I$, the points $(x_i^j)_{j\in J_i}$ are pairwise distinct and the weights $\nu_i^j$ are positive. Then, it follows from Proposition \ref{prop:characterization} that $\rhob = \bary_{\lambda}(P)$ has the form
	\[
	\rhob(x) = \exp\Bigl(-\frac{1}{2\lambda} \norm{x}^2 + \frac{1} {\lambda} \sum_{i\in I} p_i \varphi_i (x)\Bigr)
	\]
	where $\nabla \varphi_i$ is the optimal transport between $\rhob$ and $\nu_i$ so that $\varphi_i$ takes the form
	\[
	\varphi_i(x) = \max_{j\in J_i} \{x \cdot x_i^j - \psi_i^j\} \eqset \varphi_{\psi_i}(x)
	\]
	where the $\psi_i = (\psi_i^j)_{j\in J_i}$ should match the mass conservation condition, i.e.\ be such that 
	\begin{equation}\label{laguerrebar}
		\nu_i^j =\rhob(\partial \varphi_{\psi_i}^{*}(x_i^j)), \; \forall i\in I, \; \forall j\in J_i.
	\end{equation}
	In the semi-discrete optimal terminology, $\partial \varphi_{\psi_i}^{*}(x_i^j)$ is the so-called Laguerre cell where $\varphi_{\psi_i}$ coincides with $x \mapsto x \cdot x_i^j - \psi_i^j$. Computing $\rhob = \bary_{\lambda}(P)$ therefore amounts to finding $\{\psi_i^j, \; i\in I, \; j\in J_i\}$ such that \eqref{laguerrebar} holds for $\rhob$ depending on the $\psi_i^j$ as well:
	\begin{equation}\label{psibarf}
		\rhob(x) = \exp\Bigl( -\frac{1}{2\lambda} \norm{x}^2 + \frac{1} {\lambda} \sum_{i\in I} p_i \max_{j\in J_i} \{x \cdot x_i^j - \psi_i^j\} \Bigr).
	\end{equation}
	Using results from \cite{kmt} concerning the differentiability of the Kantorovich functional in the semi-discrete case, it is easy to see that the nonlinear system \eqref{laguerrebar}--\eqref{psibarf} is the system of Euler--Lagrange equations for the finite-dimensional concave maximization problem
	\begin{equation*}\label{dualpsi}
		\sup_{\psi_i^j} \left[- \sum_{i\in I} p_i \sum_{j\in J_i} \psi_i^j \nu_i^j - \lambda \int_{\R^d} \exp\Bigl( -\frac{1}{2\lambda} \norm{x}^2 + \frac{1} {\lambda} \sum_{i\in I} p_i \max_{j\in J_i} \{x \cdot x_i^j - \psi_i^j\} \Bigr) \d x\right]
	\end{equation*}
	which is the dual of the entropic barycenter problem in this semi-discrete setting.
\end{example}

\section{Properties of  the entropic barycenter}\label{sec-properties}

\subsection{Global bounds}

The aim of this paragraph is to emphasize some global bounds on the entropic barycenter which hold in the case where $\Omega$ may be unbounded, in particular it covers the case of the whole space.

\begin{lemma}\label{fisherlemma}
	The entropic-Wasserstein barycenter $\rhob$ of $P$ enjoys the following bound on the Fisher information:
	\[
	\int_{\Omega} \norm*{\nabla \log(\rhob)}^2 \rhob 
	\le \frac{1}{\lambda^2} \int_{\P_2(\R^d)} W^2(\rhob, \nu) \d P(\nu).
	\]
	In particular, $\sqrt{\rhob} \in H^1(\Omega)$, hence in case $\Omega = \R^d$ it holds that $\rhob \in L^{\infty}(\R)\cap C^{0,1/2}(\R)$ if $d = 1$, $\rhob \in L^{q}(\R^2)$ for every $q \in [1,+\infty)$ if $d = 2$ and $\rhob \in L^{\frac{d}{d-2}}(\R^d)$ if $d \geq 3$. 
	Finally, $\left(1 + \norm{x}\right) \nabla \rhob \in L^1(\R^d)$.
\end{lemma}

\begin{proof}
	According to Proposition~\ref{prop:characterization}
	\[
	\nabla \log(\rhob(x)) 
	= \frac{1}{\lambda} \int_{\P_2(\R^d)} \left(\nabla \varphi_{\rhob}^\nu(x) - x\right) \d P(\nu)
	= - \frac{1}{\lambda} \int_{\P_2(\R^d)} \nabla u_{\rhob}^\nu(x) \d P(\nu),
	\]
	thus
	\[
	\int_{\Omega} \frac{\norm*{\nabla \rhob}^2}{\rhob} = \int_{\Omega} \norm*{\nabla \log(\rhob)}^2 \rhob
	\le \frac{1}{\lambda^2} \int_{\Omega} \rhob(x) \int_{\P_2(\R^d)} \norm*{\nabla u_{\rhob}^\nu(x)}^2 \d P(\nu) \d x,
	\]
	and using Fubini's Theorem, we get that 
	\[
	\int_{\Omega} \frac{\norm*{\nabla \rhob}^2}{\rhob} 
	\le \frac{1}{\lambda^2} \int_{\P_2(\R^d)} \left[\int_{\Omega} \norm*{\nabla u_{\rhob}^\nu}^2 \rhob\right] \d P(\nu)
	= \frac{1}{\lambda^2} \int_{\P_2(\R^d)} W_2^2(\rhob, \nu) \d P(\nu).
	\]
	Finally, $\left(1 + \norm{x}\right) \nabla \rhob = 2 \left(1 + \norm{x}\right) \sqrt{\rhob} \nabla \sqrt{\rhob}$ belongs to $L^1(\R^d)$ since both $\left(1 + \norm{x}\right) \sqrt{\rhob}$ and $\nabla \sqrt{\rhob}$ are in $L^2(\R^d)$.
\end{proof}

\begin{proposition}\label{higherfisher}
	Let $p \ge 1$, and assume that
	\begin{equation}\label{finitepmoment}
		\int_{\P_2(\R^d)} m_p(\nu) \d P(\nu) < +\infty
	\end{equation}
	(where $m_p(\nu) \eqset \int_{\R^d} \norm{x}^p \d \nu(x)$). Then the entropic-Wasserstein barycenter $\rhob$ of $P$ satisfies $m_p(\rhob) < +\infty$, and more precisely, for any $r > 0$ it holds that
	\begin{equation}\label{eq:moment_bound}
		m_p(\rhob) \le \frac{6^p}{2} \left(r^p + \int_{\P_2(\R^d)} m_p(\nu) \d P(\nu)\right) + \frac{\abs{B_1(0)} \,\Gamma\!\left(\frac{d + p}{2}\right)}{2 \abs{\Omega \cap B_r(0)}} (96 \lambda)^{(d+p)/2}.
	\end{equation}
	In particular, if $\Omega = \R^d$, then
	\begin{equation}\label{eq:moment_bound_Rd}
		m_p(\rhob) \le \frac{6^p}{2} \int_{\P_2(\R^d)} m_p(\nu) \d P(\nu) + (3456 \lambda)^{p/2} \,\Gamma\!\left(\frac{d + p}{2}\right).
	\end{equation}
\end{proposition}

\begin{proof}
	Fix $r > 0$ s.t.\ $\abs*{\Omega \cap B_r(0)} > 0$ and denote $S \eqset \Omega \cap B_r(0)$. 
	Now let us take $R > 0$ and consider the set 
	\begin{equation}\label{def:Q_R}
		Q_R \eqset \left\{x \in B_R(0) \setminus B_{R/2}(0) : \norm{x} \ge 3 \left(\E \norm*{T_{\rhob}^\nu(x)} + r\right)\right\}
	\end{equation}
	(here and after, expectations are taken  w.r.t.\ $\nu \sim P$). Assume $\rhob(Q_R) > 0$ and define 
	\[
	\rho_t \eqset \rhob + t \left(\frac{\rhob(Q_R)}{\abs{S}} \ind_S - \rhob \ind_{Q_R}\right) \in \P_2(\Omega), \quad 0 \le t \le 1.
	\]
	Then 
	\begin{align*}
		\left.\frac{d}{d t} \Ent_\Omega(\rho_t)\right|_{t = 0^+} 
		&= \frac{\rhob(Q_R)}{\abs{S}} \int_S \log \rhob - \int_{Q_R} \rhob \log \rhob \\
		&\le \rhob(Q_R) \log\left(\frac{\rhob(S)}{\abs{S}}\right) - \rhob(Q_R) \log\left(\frac{\rhob(Q_R)}{\abs{Q_R}}\right) \\
		&\le \rhob(Q_R) \log\left(\frac{\abs{Q_R} \rhob(S)}{\rhob(Q_R) \abs{S}}\right) \\
		&\le \rhob(Q_R) \log\left(\frac{\abs{B_R(0)}}{\rhob(Q_R) \abs{S}}\right) \\
		&= \rhob(Q_R) \log\left(\frac{V_d R^d}{\rhob(Q_R) \abs{S}}\right),
	\end{align*}
	where $V_d \eqset \abs{B_1(0)}$ is the volume of a unit ball in $\R^d$.
	Furthermore, for any $\nu$ we can estimate $W_2^2(\rho_1, \nu)$ using the transport plan
	\[
	\gamma \eqset (\id, T_{\rhob}^\nu)_\# \left(\rhob \ind_{\R^d \setminus Q_R}\right) + \frac{1}{\abs{S}} \ind_S \otimes \left(T_{\rhob}^\nu\right)_\# \left(\rhob \ind_{Q_R}\right) \in \Pi(\rho_1,\nu),
	\]
	which gives us
	\begin{align*}
		W_2^2(\rho_1, \nu) 
		&\le \int_{\R^d \setminus Q_R} \norm*{T_{\rhob}^\nu(x) - x}^2 \rhob + \fint_S \left[\int_{Q_R} \norm*{T_{\rhob}^\nu(x) - y}^2 \rhob(x)\right] \d y \\
		&\le W_2^2(\rho_0, \nu) + \int_{Q_R} \left[\left(r + \norm*{T_{\rhob}^\nu(x)}\right)^2 - \norm*{T_{\rhob}^\nu(x) - x}^2\right] \rhob \\
		&\le W_2^2(\rho_0, \nu) + \int_{Q_R} \left[r^2 + 2 r \norm*{T_{\rhob}^\nu(x)} - \norm{x}^2 + 2 \norm*{T_{\rhob}^\nu(x)}  \norm{x}\right] \rhob.
	\end{align*}
	Then it is easy to see that, due to convexity of $W_2^2(\cdot, \cdot)$,
	\begin{align*}
		\left.\frac{d}{d t} \E W_2^2(\rho_t, \nu)\right|_{t = 0^+} 
		&\le \E W_2^2(\rho_1, \nu) - \E W_2^2(\rho_0, \nu) \\
		&\le \int_{Q_R} \left[r^2 + 2 \left(r + \norm{x}\right) \E \norm*{T_{\rhob}^\nu(x)} - \norm{x}^2\right] \rhob \\
		&= \int_{Q_R} \left[\left(r + \E \norm*{T_{\rhob}^\nu(x)}\right)^2 - \left(\norm{x} - \E \norm*{T_{\rhob}^\nu(x)}\right)^2\right] \rhob \\
		&\le - \frac{1}{3} \int_{Q_R} \norm{x}^2 \rhob
		\le - \frac{\rhob(Q_R) R^2}{12}.
	\end{align*}
	Therefore,
	\[
	\left.\frac{d}{d t} V_{P, \lambda, \Omega}(\rho_t)\right|_{t = 0^+} 
	\le \lambda \rhob(Q_R) \left(\log\left(\frac{V_d R^d}{\rhob(Q_R) \abs{S}}\right) - \frac{R^2}{24 \lambda}\right).
	\]
	On the other hand, by optimality this derivative should be nonnegative, thus
	\begin{equation}\label{eq:Q_R_bound}
		\rhob(Q_R) \le \frac{V_d R^d}{\abs{S}} \exp\left(- \frac{R^2}{24 \lambda}\right).
	\end{equation}
	
	Now we set $R_n = 2^n$ and define $q_n \eqset \rhob(Q_{R_n})$, $n \in \Z$. Note that by the definition~\eqref{def:Q_R} of $Q_R$, if $x \in \Omega \setminus \bigcup_{n \in \Z} Q_{R_n}$, then $\norm{x} < 3 \left(\E \norm*{T_{\rhob}^\nu(x)} + r\right)$. Consequently,
	\begin{equation}\label{eq:p_moment_bound}
		m_p(\rhob) = \int_\Omega \norm{x}^p \rhob 
		\le \int_{\Omega \setminus \bigcup_{n \in \Z} Q_{R_n}} 3^p \left(\E \norm*{T_{\rhob}^\nu} + r\right)^p \rhob + \sum_{n \in \Z} R_n^p q_n.
	\end{equation}
	Using the fact that $(a + b)^p \le 2^{p-1} (a^p + b^p)$, $\left(T_{\rhob}^\nu\right)_\# \rhob = \nu$, and Jensen's inequality, one can bound the first term on the r.h.s.\ as follows:
	\[
	\int_{\Omega \setminus \bigcup_{n \in \Z} Q_{2^n}} 3^p \left(\E \norm*{T_{\rhob}^\nu(x)} + r\right)^p \rhob
	\le \frac{6^p}{2} \left(r^p + \E \int_\Omega \norm*{T_{\rhob}^\nu}^p \rhob\right)
	= \frac{6^p}{2} \left(r^p + \E m_p(\nu)\right).
	\]
	Now let us bound the second term: due to~\eqref{eq:Q_R_bound} we get
	\begin{align*}
		\sum_{n \in \Z} R_n^p q_n 
		&\le \frac{V_d}{\abs{S}} \sum_{n \in \Z} R_n^{d + p} \exp\left(- \frac{R_n^2}{24 \lambda}\right) \\
		&\le \frac{V_d}{\abs{S}} \sum_{n \in \Z} \int_{2^n}^{2^{n+1}} x^{d + p - 1} \exp\left(- \frac{x^2}{96 \lambda}\right) \d x \\
		&= \frac{V_d}{\abs{S}} \int_0^{+\infty} x^{d + p - 1} \exp\left(- \frac{x^2}{96 \lambda}\right) \d x \\
		&= \frac{V_d (96 \lambda)^{(d+p)/2}}{2 \abs{S}} \Gamma\!\left(\frac{d + p}{2}\right).
	\end{align*}
	Combining the above bounds together we obtain
	\[
	m_p(\rhob) \le \frac{6^p}{2} \left(r^p + \E m_p(\nu)\right) + \frac{V_d (96 \lambda)^{(d+p)/2}}{2 \abs{S}} \Gamma\!\left(\frac{d + p}{2}\right),
	\]
	thus the first claim follows.
	
	Finally, in case $\Omega = \R^d$, we can take $r = \frac{\sqrt{96 \lambda}}{6^{p / (p+d)}}$, then using $\abs{S} = V_d r^d$ one obtains
	\begin{align*}
		m_p(\rhob) &\le \frac{6^p}{2} \E m_p(\nu) + \left(6^{d / (p+d)} \sqrt{96 \lambda}\right)^p \frac{1 + \Gamma\!\left(\frac{d + p}{2}\right)}{2} \\
		&\le \frac{6^p}{2} \E m_p(\nu) + (3456 \lambda)^{p/2} \,\Gamma\!\left(\frac{d + p}{2}\right).
	\end{align*}
\end{proof}

\begin{remark}
	Note that \eqref{eq:moment_bound_Rd} (and thus, in some sense, \eqref{eq:moment_bound}) is an interpolation between two bounds. On the one hand, if $\lambda = 0$, then $\rhob$ is a standard Wasserstein barycenter and, due to convexity of $m_p(\cdot)$ along generalized geodesics, one gets the bound
	\[
	m_p(\rhob) \le \int_{\P_2(\R^d)} m_p(\nu) \d P(\nu).
	\]
	On the other hand, if $P$ is concentrated at the measure $\delta_0$, then $\rhob = \mathcal{N}(0, \lambda I)$ by Proposition~\ref{prop:characterization}. In this case, 
	\[
	m_p(\rhob) = \frac{(2 \lambda)^{p/2} \,\Gamma\!\left(\frac{p + d}{2}\right)}{\Gamma\!\left(\frac{d}{2}\right)},
	\]
	which coincides with the second term in the r.h.s.\ of~\eqref{eq:moment_bound_Rd} up to a constant factor to the power $p$ and a factor depending on the dimension.
\end{remark}

\begin{remark}
	Let us indicate now a more elementary approach to obtain moment bounds when $\Omega$ is convex. Let $V\colon \R^d \to \R_+$ be a convex potential such that 
	\[\int_{\P_2(\R^d)} m_V(\nu) \d P(\nu)< +\infty, \text{ where } m_V(\nu) \eqset \int_{\R^d} V(x) \d \nu(x).\]
	On the one hand, thanks to \eqref{gradopticond}, the convexity of $V$ and the fact that ${\nabla \varphi_{\rhob}^{\nu}}_\# \rhob=\nu$, we have:
	\[\int_{\Omega} V(\lambda \nabla \log \rhob(x)+x) \rhob(x) \d x \leq \int_{\P_2(\R^d)} m_V(\nu) \d P(\nu).\]
	On the other hand, again by convexity $V(\lambda \nabla \log \rhob(x)+x) \rhob(x) \geq V(x) \rhob(x) +\lambda \nabla V(x) \cdot \nabla \rhob(x)$. Integrating by parts (which can be justified if $V$ is $C^{1,1}$ and using Lemma~\ref{fisherlemma}), denoting by $n$ the outward normal to $\Omega$ on $\partial \Omega$, we thus get
	\begin{equation}\label{ippcontmv}
		\int_{\Omega} (V-\lambda \Delta V) \rhob \leq \int_{\P_2(\R^d)} m_V(\nu) \d P(\nu) -\lambda \int_{\partial \Omega} \partial_n V \rhob.
	\end{equation}
	Assuming \eqref{finitepmoment} and choosing $V(x)= \norm{x-x_0}^p$ (actually, some suitable $C^{1,1}$ approximations of $V$) with $p\geq 2$ in \eqref{ippcontmv} with $x_0 \in \Omega$, observing that $\partial_n V \geq 0$ on $\partial \Omega$ since $\Omega$ is convex, we obtain the bound
	\[\int_{\Omega} \Bigl(\norm{x-x_0}^p -\lambda p(p+d-2)\ \norm{x-x_0}^{p-2}\Bigr) \rhob(x) \d x \leq \int_{\Omega} \int_{\R^d} \norm{x-x_0}^p \d \nu(x) \d P(\nu).\]
	In particular, when $\Omega=\R^d$ or, more generally, when $\Omega$ is  convex and contains $0$, we have
	\[m_2(\rhob) \leq 2 \lambda d+ \int_{\P_2(\R^d)} m_2(\nu) \d P(\nu),\]
	and for higher moments 
	\[m_p(\rhob) \leq \lambda p(p+d-2) m_{p-2}(\rhob) + \int_{\P_2(\R^d)} m_p(\nu) \d P(\nu).\]
	Note finally that when choosing $V$ linear, the two convexity inequalities we used above are equalities, yielding 
	\[\int_{\Omega} x \rhob(x) \d x+ \lambda \int_{\partial \Omega} n \rhob= \int_{\P_2(\R^d)} \int_{\R^d} x \d \nu(x) \d P(\nu).\]
\end{remark}

\begin{corollary}\label{cor:higher_fisher}
	Under assumptions of Proposition~\ref{higherfisher} it holds that $\rhob^{1/p} \in W^{1,p}(\Omega)$.
	In particular, if $p > d$, then $\rhob \in L^{\infty}(\Omega) \cap C^{0, 1-d/p}(\Omega)$.
\end{corollary}

\begin{proof}
	Once we have a bound on $m_p(\rhob)$, the fact that $\rhob^{1/p}$ is $W^{1,p}$ can be proved as for Lemma \ref{fisherlemma}. Indeed, by the same arguments (together with the crude bound $\norm*{\nabla \varphi_{\rhob}^\nu(x) - x}^p \le 2^{p-1} \left(\norm*{\nabla \varphi_{\rhob}^\nu(x)}^p + \norm{x}^p\right)$) we arrive at
	\[p^p
	\norm*{\nabla \rhob^{1/p}}_{L^p(\Omega)}^p
	= \int_{\Omega} \frac{\norm*{\nabla \rhob}^p}{\rhob^{p-1}} 
	\le \frac{2^{p-1}}{\lambda^p} \left(\int_{\P_2(\R^d)} m_p(\nu) \d P(\nu) + m_p(\rhob)\right).
	\]
\end{proof}

\subsection{Stability}

Following \cite{gouic2015existence}, let us define the $p$-Wasserstein metric between measures on $\P_p(\R^d)$:
\begin{equation}
	\mathcal{W}_p^p(P, Q) \eqset \inf_{\varGamma \in \Pi(P, Q)} \int_{\P_p(\R^d) \times \P_p(\R^d)} W_p^p(\mu, \nu) \d \varGamma(\mu, \nu).
\end{equation}

\begin{lemma}[Stability]\label{lem:stability}
	Take $p \ge 2$ and let $\{P_n\}_{n \ge 1} \subset \P_p\left(\P_p(\R^d)\right)$, $P \in \P_p\left(\P_p(\R^d)\right)$ be s.t.\ $\mathcal{W}_p(P_n, P) \to 0$. Then for $\rhob_n = \bary_{\lambda, \Omega}(P_n)$ and $\rhob = \bary_{\lambda, \Omega}(P)$ it holds that
	\begin{align}
		W_p(\rhob_n, \rhob) &\longrightarrow 0, \label{eq:Wass_stability} \\
		\rhob_n^{1/p} &\xrightarrow{W^{1,p}(\Omega)} \rhob^{1/p}, \label{eq:Wp_stability} \\
		\log \rhob_n &\xrightarrow{W^{1,q}_{\loc}(\Omega)} \log \rhob, \quad \forall\; 1 \le q < \infty. \label{eq:loc_Sob_stability}
	\end{align}
\end{lemma}

\begin{proof}
	\textbf{Proof of~\eqref{eq:Wass_stability}.}
	Note that since $W_2(\cdot, \cdot) \le W_p(\cdot, \cdot)$ and $\mathcal{W}_2(\cdot, \cdot) \le \mathcal{W}_p(\cdot, \cdot)$, one has $\mathcal{W}_2(P_n, P) \to 0$.
	According to the proof of Proposition~\ref{existbar}, $m_2(\rhob_n)$ ($m_2(\rhob)$) are uniformly bounded, thus by~\eqref{bdentm} $\Ent_\Omega(\rhob_n)$ ($\Ent_\Omega(\rhob)$) are bounded from below. Moreover, replacement of $\Omega$ with its closure $\Omb$ does not change an entropic-Wasserstein barycenter.
	Then Theorem~5.5 from \cite{kroshnin2018frechet} implies that $W_2(\rhob_n, \rhob) \to 0$.
	
	Arguing in the same way as in the proof of Proposition~\ref{higherfisher}, one can show that for any $R > 0$
	\begin{multline}\label{eq:tail_moment}
		\int_{\{\norm{x} \ge R\}} \norm{x}^p \rhob_n
		\le C \biggl[\int_{\P_p(\R^d)} \int_{\{\norm{x} \ge R\}} \left(1 + \norm[\big]{\nabla \varphi_{\rhob_n}^\nu}^p\right) \rhob_n \d P_n(\nu) \\
		+ \int_R^{+\infty} x^{d + p - 1} \exp\left(- \frac{x^2}{96 \lambda}\right) \d x\biggr], 
	\end{multline}
	where the constant $C$ depends solely on $\Omega$, $\lambda$, $p$, and $d$.
	
	To prove that $W_p(\rhob_n, \rhob) \to 0$, we also need the following result on continuity of optimal transport plans: once $W_2(\rho_n, \rho) \to 0$, $W_p(\nu_n, \nu) \to 0$, and there exists a unique optimal transport plan $\gamma_\rho^\nu$ from $\rho$ to $\nu$ for the quadratic cost function, one has
	\[
	J(\gamma_{\rho_n}^{\nu_n}, \gamma_\rho^\nu) \to 0,
	\]
	where $J(\cdot, \cdot)$ is the optimal transport cost for the cost function 
	\[
	c\bigl((x_1, y_1), (x_2, y_2)\bigr) = \norm{x_1 - x_2}^2 + \norm{y_1 - y_2}^p, \quad x_i, y_i \in \R^d,
	\]
	and $\gamma_{\rho_n}^{\nu_n}$ is any optimal transport plan from $\rho_n$ to $\nu_n$ for the quadratic cost function. Indeed,
	\[
	\int \norm{x - y}^2 \d \gamma_{\rho_n}^{\nu_n} = W_2^2(\rho_n, \nu_n) 
	\to W_2^2(\rho, \nu) = \int \norm{x - y}^2 \d \gamma_{\rho}^{\nu},
	\]
	then $\gamma_{\rho_n}^{\nu_n} \rightharpoonup \gamma_\rho^\nu$ due to Prokhorov's theorem and uniqueness of the optimal transport plan; moreover,
	\begin{multline*}
		\int_{\R^d \times \R^d} \left(\norm{x}^2 + \norm{y}^p\right) \d \gamma_{\rho_n}^{\nu_n}(x, y) = m_2(\rho_n) + m_p(\nu_n) \\
		\to m_2(\rho) + m_p(\nu) = \int_{\R^d \times \R^d} \left(\norm{x}^2 + \norm{y}^p\right) \d \gamma_{\rho}^{\nu}(x, y),
	\end{multline*}
	thus the convergence follows from \cite[Theorem~3.7]{kroshnin2018frechet}.
	Further, using \cite[Theorem~3.7]{kroshnin2018frechet} again, it is easy to see that for any closed set $G \subset \R^d$ the function
	\[
	(\rho, \nu) \mapsto \int_{G} \left(1 + \norm*{\nabla \varphi_{\rho}^\nu}^p\right) \rho = \int_{G \times \R^d} \left(1 + \norm{y}^p\right) \d \gamma_\rho^\nu(x, y)
	\]
	is upper-semicontinuous w.r.t.\ convergence in $W_2$ distance (for $\rho$) and $W_p$ distance (for $\nu$), as well as its average w.r.t.\ a measure on $\P_p(\R^d)$:
	\[
	(\rho, P) \mapsto \int_{\P_p(\R^d)} \int_{G} \left(1 + \norm*{\nabla \varphi_{\rho}^\nu}^p\right) \rho \d P(\nu).
	\]
	Hence for all $R > 0$ one obtains
	\begin{multline*}
		\limsup \int_{\P_p(\R^d)} \int_{\{\norm{x} \ge R\}} \left(1 + \norm*{\nabla \varphi_{\rhob_n}^\nu}^p\right) \rhob_n \d P_n(\nu) \\ 
		\le \int_{\P_p(\R^d)} \int_{\{\norm{x} \ge R\}} \left(1 + \norm*{\nabla \varphi_{\rhob}^\nu}^p\right) \rhob \d P(\nu).
	\end{multline*}
	Using this together with~\eqref{eq:tail_moment}, we get that
	\begin{multline*}
		\limsup \int_{\{\norm{x} \ge R\}} \norm{x}^p \rhob_n 
		\le C \biggl[\int_{\{\norm{x} \ge R\}} \left(\int_{\P_p(\R^d)} \left(1 + \norm*{\nabla \varphi_{\rhob}^\nu}^p\right) \d P(\nu)\right) \rhob \\
		+ \int_R^{+\infty} x^{d + p - 1} \exp\left(- \frac{x^2}{96 \lambda}\right) \d x\biggr] \to 0 \text{ as } R \to 0.
	\end{multline*}
	Thus the measures $\norm{\cdot}^p \rhob_n$ are uniformly integrable, and by the criterion of convergence in a Wasserstein space (see e.g.\ \cite[Theorem~6.9]{Villani1} or \cite[Theorem~3.7]{kroshnin2018frechet}), we deduce  that $W_p(\rhob_n, \rhob) \to 0$.
	
	\textbf{Proof of~\eqref{eq:Wp_stability} and~\eqref{eq:loc_Sob_stability}.}
	Fix an arbitrary open set $U \subset\subset \Omega$.
	By Lemma~\ref{harnacklem} 
	\begin{align*}
		\lip\left(\varphi_{\rhob_n}^\nu\big|_U\right)
		&\le \frac{C}{\inf\limits_{x \in U_{3 r / 4}} \rhob_n\left(B_{r / 4}(x)\right)} \left(\iom \norm{\nabla \varphi_{\rhob_n}^\nu}^2 \rhob_n\right)^{1/2} \\
		&= \frac{C}{\inf\limits_{x \in U_{3 r / 4}} \rhob_n\left(B_{r / 4}(x)\right)} \sqrt{m_2(\nu)},
	\end{align*}
	where $r = d(U, \partial \Omega)$. 
	Since $\rhob_n \rightharpoonup \rhob$ and $\rhob > 0$ on $\Omega$, we have $\inf\limits_{x \in U_{3 r / 4}} \rhob_n\left(B_{r / 4}(x)\right) \ge c > 0$ for any $n$.
	Therefore, the  functions
	\[
	\bar\varphi_n = \lambda \log \rhob_n + \frac{\norm{\cdot}^2}{2}
	= \int_{\P_2(\R^d)} \varphi_{\rhob_n}^\nu \d P_n(\nu)
	\]
	are uniformly Lipschitz continuous on $U$ for all $n$ since $\int_{\P_2(\R^d)} m_2(\nu) \d P_n(\nu)$ are uniformly bounded.
	Furthermore, as $\rhob_n \rightharpoonup \rhob>0$, $\bar\varphi_n$ are also uniformly bounded on $U$. Then, by the Arzel{\`a}--Ascoli theorem, $\bar\varphi_n \xrightarrow{C(U)} \bar\varphi$, and we deduce from weak convergence that $\bar\varphi = \lambda  \log \rhob + \frac{\norm{\cdot}^2}{2}$. Moreover, every $\bar\varphi_n$ is convex, thus $\nabla \bar\varphi_n \to \nabla \bar\varphi$ a.e.\ on $U$.
	Hence, by Lebesgue's dominated convergence theorem, we get $\bar\varphi_n \xrightarrow{W^{1,q}(U)} \bar\varphi$ for any $1 \le q < \infty$ and thus \eqref{eq:loc_Sob_stability}.
	
	Further, using~\eqref{gradopticond}, we get
	\[
	\int_{\Omega \setminus U} \norm*{\nabla \rhob_n^{1/p}}^p 
	= \frac{1}{p^p} \int_{\Omega \setminus U} \norm*{\nabla \log \rhob_n}^p \rhob_n
	\le \frac{2^{p-1}}{(p \lambda)^p} \int_{\Omega \setminus U} \left(\norm*{\nabla \bar\varphi_n}^p + \norm{x}^p\right) \rhob_n.
	\]
	Since the functions $\rho \mapsto \int_{\Omega \setminus U} \norm{x}^p \rho$ and $(\rho, P) \mapsto \int_{\P_p(\R^d)} \int_{\Omega \setminus U} \norm*{\nabla \varphi_{\rho}^\nu}^p \rho \d P(\nu)$ are u.s.c., we obtain that
	\[
	\limsup \int_{\Omega \setminus U} \norm*{\nabla \rhob_n^{1/p}}^p \to 0 \quad\text{as}\quad U \to \Omega
	\]
	(e.g.\ in a sense that $\rhob(\Omega \setminus U) \to 0$).
	Finally, this together with~\eqref{eq:loc_Sob_stability} yields that $\rhob_n^{1/p} \xrightarrow{W^{1,p}(\Omega)} \rhob^{1/p}$.
\end{proof}

In particular, the previous lemma shows that one can approximate the barycenter $\rhob$ by approximating $P$ with discrete measures supported on some dense set of measures, e.g.\ discrete or having smooth densities.
As another corollary of Lemma~\ref{lem:stability}, in section~\ref{sec-CLT} we will obtain a law of large numbers for entropic-Wasserstein barycenters.

\subsection{A maximum principle} 

\begin{proposition}\label{mpprop}
	Assume that $\Omega$ is convex and $P\bigl(\{\nu(\Omega) = 1\}\bigr) = 1$, and let $\rhob \eqset \bary_{\lambda, \Omega}(P)$ be its entropic barycenter. Then 
	\[
	\norm{\rhob}_{L^{\infty}(\R^d)} \le \left(\int_{\P_2(\R^d)} \norm{\nu}_{L^{\infty}(\R^d)}^{-1/d} \d P(\nu) \right)^{-d} .
	\]
\end{proposition}

\begin{proof}
	We first prove the result in the simple case where $P$ is supported by finitely many measures and then proceed by approximation thanks to the stability Lemma \ref{lem:stability} (more precisely, its corollary Theorem~\ref{thm:LLN}).
	\smallskip
	
	\textbf{Step 1: the case of finitely many measures.}
	
	Fix a compact convex set $K \subset \Omega$ with nonempty interior. Assume that $P = \sum_{i = 1}^N p_i \delta_{\nu_i}$, where each $\nu_i$ is supported in $K$ and has a $C^{0, \alpha}$, bounded away from $0$ density on $K$. Since $K$ is bounded, all $\varphi_{\rhob}^{\nu_i}$ are Lipschitz, so we can take the continuous version of $\rhob$ on $\Omb$. 
	Now fix an arbitrary $x \in \Omb \setminus K$. Since $\rhob > 0$ on $\Omega$ and $\left(\nabla \varphi_{\rhob}^{\nu_i}\right)_\# \rhob = \nu_i$ for all $i$, there are subgradients $\nabla \varphi_{\rhob}^{\nu_i}(x) \in K$. Let $y = \sum_{i = 1}^N p_i \nabla \varphi_{\rhob}^{\nu_i}(x) \in K$, $v = \frac{y - x}{\norm{y - x}}$, then thanks to \eqref{densrhob}
	\[
	\partial_{v} \log \rhob(x) \ge \frac{1}{\lambda} \langle y - x, v \rangle = \frac{1}{\lambda} \norm{y - x} > 0,
	\]
	therefore $x$ cannot be a maximum point of $\rhob$, and $\rhob$ actually attains its maximum on $K$.
	
	Further, since $\log(\rhob) \in W^{1, \infty}_\loc(\Omega)$, the regularity result of Cordero-Erausquin and Figalli \cite{CorderoFigalli} yields that $\varphi_{\rhob}^{\nu_i}$ is in fact $C^{2, \alpha}_{\loc}$. Then at its maximum point $x \in \Omega$ we should have, on the one hand
	\[
	\sum_{i=1}^N p_i D^2 \varphi_{\rhob}^{\nu_i}(x) \leq I .
	\]
	On the other hand, using the Monge--Amp\`ere equation $\rhob = \det\left(D^2 \varphi_{\rhob}^{\nu_i}\right) \nu_i \left(\nabla \varphi_{\rhob}^{\nu_i}\right)$ (see also \eqref{eq:MAclassic}), we get
	\[
	\rhob(x) \leq \norm{\nu_i}_{L^{\infty}(\R^d)} \det\left(D^2 \varphi_{\rhob}^{\nu_i}(x)\right), \quad i = 1, \ldots, N.
	\]
	So, using the concavity of $\det(\cdot)^{1/d}$ over symmetric positive semi-definite matrices, we obtain 
	\begin{align*} 
		\sum_{i = 1}^N p_i \left(\frac{\rhob(x)}{\norm{\nu_i}_{L^{\infty}(\R^d)}}\right)^{1/d} &\le \sum_{i = 1}^{N} p_i \det\left(D^2 \varphi_{\rhob}^{\nu_i}(x)\right)^{1/d} \\
		&\le \det\left(\sum_{i=1}^{N} p_i D^2 \varphi_{\rhob}^{\nu_i}(x)\right)^{1/d} \le 1,
	\end{align*}
	what gives 
	\[
	\rhob \le \left(\sum_{i = 1}^N p_i \norm{\nu_i}_{L^{\infty}(\R^d)}^{-1/d}\right)^{-d} .
	\]
	
	Of course, the requirement that $\nu_i$ is bounded away from $0$ is just here to justify twice differentiability of $\varphi_{\rhob}^{\nu_i}$, if we drop this assumption replacing $\nu_i$ by $\nu_i^n = (1 - \frac{1}{n}) \nu_i + \frac{1}{n |K|}$, using Lemma~\ref{lem:stability}, we get the same conclusion by letting $n \to \infty$. In a similar way, H\"{o}lder regularity of the $\nu_i$'s can also be removed by suitably mollifying these measures and arguing by stability again. 
	Finally, if $P = \sum_{i=1}^N p_i \delta_{\nu_i}$ with $m_2(\nu_i) < +\infty$, we can find an increasing sequence of compact convex sets $K_n \subset \Omega$, such that for every $n \in \N$
	\[
	\max_{i=1, \ldots, N} \int_{\R^d \setminus K_n} \left(1 + \norm{x}^2\right) \d \nu_i(x) \leq \frac{1}{n}.
	\]
	Set
	\[
	\nu_i^n \eqset \frac{\nu_i \ind_{K_n}}{\nu_i(K_n)} \le \frac{n}{n - 1} \nu_i, \quad 
	P_n \eqset \sum_{i=1}^N p_i \delta_{\nu_i^n},
	\]
	then $\rhob_n \eqset \bary_{\lambda, \Omega}(P_n)$ is bounded with 
	
	\[
	\rhob_n \le \frac{n}{n-1} \left(\sum_{i = 1}^N p_i \norm{\nu_i}_{L^{\infty}(\R^d)}^{-1/d}\right)^{-d}.
	\]
	Since $W_2(\nu_i^n, \nu_i) \to 0$ for all $1 \le i \le N$, we have $\mathcal{W}_2^2(P_n, P) \to 0$, thus stability enables us to conclude that 
	
	\[
	\rhob \leq \left(\sum_{i = 1}^N p_i \norm{\nu_i}_{L^{\infty}(\R^d)}^{-1/d}\right)^{-d} .
	\]
	
	\smallskip
	
	\textbf{Step 2: the general case.}
	
	We now consider the case of a general Borel probability $P$ on $\P_2(\R^d)$ satisfying \eqref{finitemomentP} 
	and concentrated on measures giving full mass to $\Omega$.
	Let $\nu_1, \nu_2, \dots$ be i.i.d.\ random measures drawn from $P$. Then, by Theorem~\ref{thm:LLN}, the empirical barycenters $\rhob_n \eqset \bary_{\lambda, \Omega}(P_n)$, where $P_n \eqset \frac{1}{n} \sum_{i = 1}^n \delta_{\nu_i}$ is the empirical measure, a.s.\ converge to $\rhob$ in $2$-Wasserstein distance. 
	Since 
	\[
	\frac{1}{n} \sum_{i = 1}^n \norm{\nu_i}_{L^{\infty}(\R^d)}^{-1/d} \to \int_{\P_2(\R^d)} \norm{\nu}_{L^{\infty}(\R^d)}^{-1/d} \d P(\nu) \quad\text{a.s.}
	\]
	by the strong law of large numbers, we conclude using Step~1.
\end{proof}

\begin{remark}
	If, under the assumptions of the above proposition, 
	\[
	P\left(\{\nu \in L^{\infty}(\R^d), \; \nu \leq C\}\right) = \alpha > 0,
	\] 
	then it gives
	\[
	\norm{\rhob}_{L^{\infty}(\R^d)} \le \frac{C}{\alpha^d}.
	\]
	The same bound was obtained in Theorem~6.1 from \cite{KimPass} for $2$-Wasserstein barycenters on Riemannian manifolds.
\end{remark}

The following simple example shows that convexity of $\Omega$ is essential for the maximum principle (even if $P$-a.e.\ measure $\nu$ is concentrated on $\Omega$).

\begin{example}
	Consider the one-dimensional case where $\Omega = [-8,-4] \cup [-1,1] \cup [4,8]$. Let $P = \frac{1}{2} \delta_{\nu_-} + \frac{1}{2} \delta_{\nu_+}$, $\nu_- = \frac{1}{4} \ind_{(-8,-4)}$, $\nu_+ = \frac{1}{4} \ind_{(4,8)}$. First, we take $\lambda = 0$, thus $\rhob_0 \eqset \bary_{\Omega,0}(P)$ is an ordinary Wasserstein barycenter (constrained to be supported on $\Omega$). It is easy to see that $\rhob_0$ is actually supported on $[-1,1]$, so $\norm{\rhob_0}_{L^\infty(\Omega)} \ge \frac{1}{2}$ while $\norm{\nu_-}_{L^\infty(\Omega)} = \norm{\nu_+}_{L^\infty(\Omega)} = \frac{1}{4}$.
	Now we consider $\rhob_\lambda \eqset \bary_{\Omega,\lambda}(P)$ and let $\lambda \to 0$. By compactness, we readily get that $\rhob_\lambda \rightharpoonup \rhob_0$, so, for $\lambda$ small enough, we have $\norm{\rhob_\lambda}_{L^\infty(\Omega)} > \frac{1}{4}$.
	Finally, by rescaling, one can construct examples violating the maximum principle for any $\lambda > 0$.
\end{example}
\section{Higher regularity}\label{sec-reg}

\subsection{The bounded case}
The theory developed so far has needed very mild assumptions on $\Omega$. To deduce higher regularity (up to the boundary) of the Kantorovich potentials and the barycenter we need to impose more conditions on the domain. 

Suppose that $P$ is  concentrated on sufficiently regular probability measures supported on a closed ball of radius $R>0$, $\bar B \eqset \Omb = \bar B_R(0)$, more precisely, assume that for some $\alpha \in (0,1), k \in \N^*$ and $C > 0 $ 
\begin{align}\label{ass_reg_P}
P \Bigl( \underbrace{\left\{ \nu \in \P_{\mathrm{ac}}(\R^d): \nu(\Omb) = 1, \norm{\nu}_{C^{k,\alpha}(\Omb)} + \norm{\log \nu}_{L^\infty(\Omb)} \leq C \right\}}_{\eqqcolon \mathcal{Q}} \Bigr) = 1.
\end{align}

\begin{remark}\label{rem:GenRegBdry}
The following arguments are presented here for the case of a ball for simplicity but work for compact convex sets with $C^{k+2,\alpha}$-boundary which are strongly convex with a uniform modulus of convexity. More precisely, we require that there are $m$-strongly convex functions $H_\nu, H \in C^{k+2,\alpha}(\R^d)$ for $m>0$ such that
\begin{align*}
\Omega &= \{x \in \R^d:~ H(x) < 0\}, \partial \Omega = \{x \in \R^d:~ H(x) = 0 \}, \\
\supp \nu &= \{x \in \R^d:~ H_\nu(x) \leq 0\}, \partial (\supp \nu) = \{x \in \R^d:~ H_\nu(x) = 0 \},
\end{align*}
and there is an $R > 0$ such that $\Omega,~ \supp \nu \subset B_R(0)$ for $P$-a.e. $\nu$. We add remarks at the proofs that significantly depend on the domain.
\end{remark}

Thanks to the entropic regularization, this regularity implies regularity for the potentials and the barycenter.
\begin{proposition}\label{prop:regularity}
	Under assumption \eqref{ass_reg_P}, one has
	\begin{equation}
	\varphi_{\rhob}^{\nu}  \in C^{k+2, \alpha}(\Omb) \mbox{ for } P \mbox{-a.e. } \nu\quad \mbox{and}\quad \rhob \in C^{k+2, \alpha}(\Omb),
	\end{equation}
	and there is a constant $K > 0$ such that
	\begin{equation}\label{eq:regular_phi}
	\norm*{\varphi_{\rhob}^{\nu}}_{C^{k+2,\alpha}(\Omb)}, \norm*{\varphi_{\nu}^{\rhob}}_{C^{k+2,\alpha}(\Omb)} \leq K \mbox{ for } P \mbox{-a.e. } \nu.
	\end{equation}
	Furthermore, for $P$-a.e.\ $\nu$ the transport $\nabla \varphi_{\rhob}^{\nu}\colon \Omb \rightarrow \Omb$ is a diffeomorphism of class $C^{k+1,\alpha}$.
\end{proposition}
\begin{proof}
	By \eqref{ass_reg_P} $P$-a.e.\ $\nu \in C^{0,\alpha}(\Omb)$ is bounded from below and above on $\Omb$ by a constant only depending on $C$. 
	With the representation of $\rhob$ in \eqref{densrhob} we obtain that $\nabla \log\rhob$ is bounded by $2R/\lambda$ a.e. Together with $\int \rhob = 1$ this implies that $\norm{\log \rhob}_{C^{0,1}(\Omb)}$ is bounded by a constant only depending on $R$ and $\lambda$.
	
	This implies by Caffarelli's regularity theory for Monge--Amp\`ere equations (see \cite{caffarelli1996boundary} for the original paper and Theorem~3.3 \cite{defig2014} for a concise formulation) that for any $\nu \in \mathcal{Q}$, $\varphi_{\rhob}^{\nu} \in C^{2, \alpha}(\Omb)$ and $\nabla \varphi_{\rhob}^{\nu}\colon \Omb \rightarrow \Omb$ is a diffeomorphism. 
		
	For the uniform estimate again by Caffarelli's regularity theory for Monge--Amp\`ere equations (theorem on page~3 of \cite{caffarelli1992}) there is an $\alpha_1 \in (0,1)$ and constant $C_1$ (only depending on $\alpha_1$, $C$ and $R$) such that 
	\begin{equation*}
	\norm*{\varphi_{\rhob}^{\nu}}_{C^{1,\alpha_1}(\Omb)}, \norm*{\varphi_{\nu}^{\rhob}}_{C^{1,\alpha_1}(\Omb)} \leq C_1 \quad \mbox{ for every } \nu \in \mathcal{Q}.
	\end{equation*}
	This implies in particular $\rhob \in C^{1,\alpha_1}(\Omb)$ by \eqref{densrhob} and we can apply Theorem~\ref{thm:DiffMAmapping} to see that 
	\begin{equation*}
	\begin{array}{rll}
	\Phi^{\rhob}\colon \left\{ \nu \in C^{0, \alpha_1}(\Omb): \nu(\Omb) = 1, \norm{\log \nu}_{L^\infty(\Omb)} < \infty \right\} &\rightarrow& \cM \\
	\nu &\mapsto& \varphi_\nu^{\rhob}
	\end{array}
	\end{equation*}
	is continuous (where $\cM$ denotes the set of $C^{2, \alpha_1}(\Omb)$ convex potentials  $\varphi$ with zero mean such that $\norm{ \nabla \varphi} = R$ on $\partial \Omega$). Now note that, by the compact embedding of H\"older spaces, $\mathcal{Q}$ is compact in $C^{0, \alpha_1}(\Omb)$. This implies that $\Phi^{\rhob}( \mathcal{Q})$ is compact in $C^{2, \alpha_1}(\Omb)$. Hence, there is a $K_1 > 0$ such that
	\begin{equation}\label{eq:bound_LegendreTrans}
	\norm*{\varphi_{\nu}^{\rhob}}_{C^{2,\alpha_1}(\Omb)} \leq K_1 \mbox{ for } P \mbox{-a.e. } \nu.
	\end{equation}
	
	Furthermore, since each $\varphi_{\nu}^{\rhob}$ is strongly convex thanks to compactness of $\Phi^{\rhob}(\mathcal{Q})$ we conclude that there is constant $c>0$ such that
	\begin{equation}\label{eq:unif_conv_LegTrans}
	D^2 \varphi_{\nu}^{\rhob} \geq c \mbox{ for } P \mbox{-a.e. } \nu,
	\end{equation}
	so that we obtain
	\begin{equation}\label{eq:bound_secDer}
	\norm*{D^2 \varphi_{\rhob}^{\nu}}_{L^\infty(\Omega)} \leq c \mbox{ for } P \mbox{-a.e. } \nu,
	\end{equation}
	which gives $\rhob \in C^{1,1}(\Omb)$ and then again by Caffarelli's regularity theory for Monge--Amp\`ere equations $\varphi_{\rhob}^{\nu} \in C^{3, \alpha}(\Omb)$.
	Differentiating now the Monge--Amp\`ere equation (which is satisfied in the classical sense)
	\begin{equation*}
	\begin{array}{rlll}
	\det(D^2 \varphi_{\rhob}^{\nu}) \nu(\nabla \varphi_{\rhob}^{\nu}) &=& \rhob & \text{ in } \Omega, \\
	\norm*{\nabla\varphi_{\rhob}^{\nu}}^2 &=& R^2 & \text{ on } \partial \Omega,
	\end{array}
	\end{equation*}
	in direction $e \in \mathbb{S}^{d-1}$, we obtain by the same considerations as in the appendix
	\begin{equation}
	\begin{array}{rlll}
	\div(A_\nu \nabla (\partial_e \varphi_{\rhob}^{\nu})) &=& \partial_e \rhob & \text{ in } \Omega, \\
	\nabla \varphi_{\rhob}^{\nu} \cdot \nabla (\partial_e \varphi_{\rhob}^{\nu}) &=& 0 & \text{ on } \partial \Omega,
	\end{array}
	\end{equation}
	where $A_\nu = \nu(\nabla \varphi_{\rhob}^{\nu}) \det(D^2 \varphi_{\rhob}^{\nu}) (D^2 \varphi_{\rhob}^{\nu})^{-1}$.
	Thanks to Lemma~\ref{lem:gradient_oblique} and \eqref{eq:unif_conv_LegTrans} we can finally deduce by classical Schauder estimates (Theorem~6.30 in \cite{gilbarg2015elliptic}) that there is constant $K > 0$ uniform in $\nu$ such that 
	\begin{equation*}
	\norm*{\partial_e \varphi_{\rhob}^{\nu}}_{C^{2,\alpha}(\Omb)} \leq K \left(\norm*{\partial_e \varphi_{\rhob}^{\nu}}_{C^{0,\alpha}(\Omb)} + \norm{\partial_e \rhob}_{C^{0,\alpha}(\Omb)}\right).
	\end{equation*}
	This concludes the uniform estimate of $\varphi_{\rhob}^{\nu}$ in $C^{3,\alpha}(\Omb)$ for $P$-a.e. $\nu$, and by again employing \eqref{densrhob} we deduce $\rhob \in C^{3,\alpha}(\Omb)$. The same bound follows for $\varphi_{\nu}^{\rhob}$ by exchanging the role of $\rhob$ and $\nu$. Higher regularity follows by standard elliptic theory.

\end{proof}

Note in particular that $\varphi_{\rhob}^{\nu}$ satisfies the Monge--Amp\`ere equation, subject to the \emph{second boundary} value condition, which encodes the fact that $\nabla \varphi_{\rhob}^\nu$ maps the ball into itself, in the classical sense
\begin{equation}\label{eq:MAclassic}
\begin{aligned}
\det(D^2\varphi_{\rhob}^\nu) \nu (\nabla \varphi_{\rhob}^\nu) &= \rhob \text{ in $B$} \\
\nabla \varphi_{\rhob}^\nu (B) &\subset B,
\end{aligned}
\end{equation}
and that the second boundary value condition is equivalent (see Lemma \ref{lem:diff_boundary_equiv})  to an eikonal equation on the boundary
\begin{align*}
	\norm{\nabla \varphi_{\rhob}^\nu (x)}^2 = R^2, \; \forall x\in \partial B.
\end{align*}

\subsection{The case of log-concave measures on \texorpdfstring{$\R^d$}{Rd}}

Caffarelli's contraction principle \cite{CaffarelliFKG}, generalized by Kolesnikov in \cite{Kolesnikov}, implies global (and dimension-free) Lipschitz (or H\"{o}lder) global estimates for the optimal transport between suitable log-concave measures. In its original form, Caffarelli's Theorem says that the optimal transport between the standard Gaussian $\gamma$ and a measure which is more log-concave (i.e.\ has the form $e^{-V} \gamma$ with $V$ convex) is $1$-Lipschitz. Since the entropic barycenter is less log-concave than a Gaussian, if the measures $\nu$ satisfy a uniform log-concavity estimate, one can deduce a $C^{1,1}$ regularity result for $\log(\rhob)$:

\begin{proposition}\label{logconcaveprop}
	Assume $\Omega = \R^d$ and  that  for $P$-a.e. $\nu$ there is some $A_\nu >0$ such that $\nu$ writes as $\diff \nu = e^{-V_\nu} \diff y$ with $D^2 V_\nu \geq A_\nu I$ (in the sense of distributions), such that $\E\left[\sqrt{\lambda A_\nu}^{-1}\right]<\infty$, where the expectation is taken with respect to the random variable $\nu$ distributed according to $P.$ Let $\rhob \eqset \bary_{\lambda}(P)$ be its entropic barycenter. Then $\log \rhob \in C^{1, 1}(\R^d)$ and more precisely there holds
	\begin{equation}\label{c11logrhob}
	-I \leq \lambda D^2 \log \rhob \leq \Bigl( \E\left[\frac{1}{\sqrt{\lambda A_\nu}}\right] - 1 \Bigr) I. 
	\end{equation} 
\end{proposition}
\begin{proof}
	It directly follows from \eqref{densrhob} that $\rhob=e^{-\psi}$ with $D^2 \psi \leq \frac{I}{\lambda}$. Since for $P$-a.e. $\nu$ we have $\diff \nu=e^{-V_\nu} \diff y$ with $D^2 V_\nu \geq A_\nu I$, thanks to Caffarelli's contraction Theorem, the optimal transport map $\nabla \varphi_{\rhob}^\nu$ is Lipschitz with the explicit estimate
	\[0 \leq D^2 \varphi_{\rhob}^\nu \leq \frac{I}{\sqrt{\lambda A_\nu}} \quad P\text{-a.e.}\]
	so that the convex potential $\Phi \eqset \E\left[ \varphi_{\rhob}^{\nu} \right]$ is $C^{1,1}$ and has $\E\left[\sqrt{\lambda A}^{-1}\right]<\infty$ as an upper bound on its Hessian. Since $\lambda \nabla \log(\rhob) + \id = \nabla \Phi$, the bound \eqref{c11logrhob} directly follows. 
\end{proof}
\section{Statistical properties}\label{sec-CLT}

\subsection{Stochastic setting and law of large numbers}

Now we consider the following stochastic setting \cite{bigot2012consistent,kroshnin2018frechet,ahidar2019convergence}: let $P$, as above, be a distribution on $\P_2(\Omega)$ with finite second moment, and $\nu_1, \nu_2, \dots$ be independent random measures drawn from $P$. 
We will call the barycenter of the first $n$ measures $\nu_1, \dots, \nu_n$ an \emph{empirical barycenter}: $\rhob_n = \bary_{\lambda, \Omega}(P_n)$, where $P_n = \frac{1}{n} \sum_{i = 1}^n \delta_{\nu_i}$ is the empirical measure.
Note that $\rhob_n$ is random, and in this section we will establish its statistical properties, namely, consistency and (under additional assumptions) a central limit theorem.
As already mentioned in section~\ref{sec-properties}, a LLN follows immediately from Lemma~\ref{lem:stability}.

\begin{thm}[law of large numbers]\label{thm:LLN}
	Assume $\int_{\P_2(\R^d)} m_p(\nu) \d P(\nu) < +\infty$ for some $p \ge 2$. Let $\rhob$ be the entropic-Wasserstein barycenter of $P$ and $\{\rhob_n\}_{n \in \N}$ be empirical barycenters. Then it a.s.\ holds that
	\begin{align*}
		W_p(\rhob_n, \rhob) &\longrightarrow 0, \\
		\log \rhob_n &\xrightarrow{W^{1,q}_{\loc}(\Omega)} \log \rhob \quad \forall 1 \le q < \infty, \\
		\rhob_n^{1/p} &\xrightarrow{W^{1,p}(\Omega)} \rhob^{1/p}.
	\end{align*}
	Moreover, under assumption~\eqref{ass_reg_P} $\rhob_n \xrightarrow{a.s.} \rhob$ in $C^{k+2,\beta}(\Omb)$ for any $\beta \in (0, \alpha)$.
\end{thm}

\begin{proof}
	It is well-known that, since the Wasserstein space is Polish, empirical measures $P_n$ converge to $P$ in $\mathcal{W}_p$ metric (see e.g.\ Corollary~5.9 in \cite{kroshnin2018frechet}). Then the first part of the theorem follows from Lemma~\ref{lem:stability}.
	
	Further, once~\eqref{ass_reg_P} holds, sequence $\{\rhob_n\}_{n \in \N}$ is uniformly bounded in $C^{k+2,\alpha}(\Omb)$ by Proposition~\ref{prop:regularity}. Therefore, due to compact H{\"o}lder embedding and weak convergence $\rhob_n \rightharpoonup \rhob$, the second claim follows.
\end{proof}

\subsection{Central limit theorem}

Let $H$ be a separable Hilbert space endowed with its Borel sigma-algebra.
Recall that random variables $\{X_n\}_{n \in \N}$ taking values in $H$ converge in distribution to a random variable\ $X$ if $\E f(X_n) \to \E f(X)$ for any bounded continuous function $f$ on $H$. We denote this convergence by
\[
X_n \xrightarrow{d} X.
\]
We also need to recall the notion of strong operator topology (SOT): operators $A_n$ on $H$ converge to $A$ in SOT ($A_n \xrightarrow{\mathrm{SOT}} A$), if $A_n u \to A u$ for all $u \in H$.
Finally, to prove a central limit theorem for barycenters we will use some technical results from probability theory postponed to Appendix~\ref{sec-app_prob}.

Let us also introduce the following notation:
if $\mathcal{F}$ is a space of integrable functions on $\Omb$, then
\begin{equation}\label{def:zero_mean_space}
	\mathcal{F}_\diamond \eqset \left\{f \in \mathcal{F} : \int_\Omega f = 0\right\}.
\end{equation}

\begin{thm}[central limit theorem]\label{thm:CLT}
	Let assumption~\eqref{ass_reg_P} be fulfilled with $k = 1$. Then a CLT for empirical barycenters holds in $H^2_\diamond(B)$:
	\[
	\sqrt{n} \left(\rhob_n - \rhob\right) \xrightarrow{d} \xi \sim \mathcal{N}(0, \Sigma),
	\]
	with covariance operator $\Sigma = G^{-1} \Var(\varphi_{\rhob}^\nu) G^{-1}$ where 
	\[
	G \colon u \mapsto \lambda \frac{u}{\rhob} - \lambda \fint_B \frac{u}{\rhob} - \E (\Phi^\nu)'(\rhob)
	\]
	and $\Phi^{\nu}(\rhob)$ is the zero-mean Brenier potential between $\rhob$ and $\nu$. 
\end{thm}

\begin{proof}
	\textbf{Step 1.}
	Let us introduce the following map $F$ 
	\begin{align*}
	F  & :   \left\{\rho \in C^2(\bar{B}) : \int_B \rho = 1,\, \min_{\bar{B}} \rho > 0\right\}  \to C^2_{\diamond}(\bar{B}) \\
	  & \rho  \mapsto \lambda \log \rho + \frac{\norm{\cdot}^2}{2} - \fint_B \left(\lambda \log \rho(x) + \frac{\norm{x}^2}{2}\right).
	\end{align*}
	It is continuously differentiable and its derivative is
	\[
	F'(\rho) \colon u \mapsto \lambda \frac{u}{\rho} - \lambda \fint_B \frac{u}{\rho}.
	\]
	
	Then equation~\eqref{densrhob} can be rewritten (see Appendix \ref{sec-app} for  properties of the map $\Phi^{\nu}$) as follows:
	\[
	F(\rhob) = \E \Phi^\nu(\rhob).
	\]
	Respectively, for the empirical barycenter it reads as
	\[
	F(\rhob_n) = \frac{1}{n} \sum_{i = 1}^n \Phi^{\nu_i}(\rhob_n).
	\]
	Combining the above equations and using differentiability of $F$ and $\Phi^\nu$ (Theorem~\ref{thm:DiffMAmapping}), we obtain
	\begin{equation}\label{eq:char_linear}
	    G_n(\rhob_n - \rhob) = F(\rhob_n) - F(\rhob) - \frac{1}{n} \sum_{i = 1}^n \left(\Phi^{\nu_i}(\rhob_n) - \Phi^{\nu_i}(\rhob)\right) = \frac{1}{n} \sum_{i = 1}^n \varphi_i - \E \varphi,
	\end{equation}
	where $\varphi_i = \varphi^{\nu_i}_{\rhob}$, $\E \varphi = \E \varphi^\nu_{\rhob}$, and the operator $G_n$ is defined as follows:
	\begin{equation}
	    G_n \eqset \int_0^1 F'(\rhob_n^t) \d t - \frac{1}{n} \sum_{i = 1}^n \int_0^1 (\Phi^{\nu_i})'(\rhob_n^t) \d t
	\end{equation}
	with $\rhob_n^t \eqset (1 - t) \rhob + t \rhob_n$.

	\textbf{Step 2.}
	We are going to apply a delta-method to prove a CLT and to do this we need a convergence (in an appropriate space)
	\begin{equation}
		(G_n)^{-1} \xrightarrow{P} G^{-1}, \quad 
		G \eqset F'(\rhob) - \E (\Phi^\nu)'(\rhob).
	\end{equation}
	
	We will consider CLT in $H^2_\diamond(B)$, but first, let us extend all the linear operators above to $L^2_\diamond(B)$.
	Denote by $\Bary_{\lambda, B}(\mathcal{Q})$ the set of entropic barycenters of all measures supported on $\mathcal{Q}$: 
	\[
	\Bary_{\lambda, B}(\mathcal{Q}) \eqset \left\{\bary_{\lambda, B}(P) : P \in \P(\P_2(\R^d)),\, P(\mathcal{Q}) = 1\right\}.
	\]
	Clearly, the operators $F'(\rho)$ are Hermitian, bounded and positive definite for all $\rho \in \Bary_{\lambda, B}(\mathcal{Q})$. 
	Using \eqref{eq:regular_phi} and \eqref{densrhob} we conclude that these $\rho$ are uniformly bounded away from zero, thus $F'(\rho)$ are uniformly positive-definite: indeed, for any $u \in L^2_\diamond(B)$
	\begin{equation}\label{eq:diff_F_bound}
		\langle u, F'(\rho) u \rangle_{L^2(B)} = \int_{B} \lambda \frac{u^2}{\rho} 
		\ge \frac{\lambda}{\min_{\bar{B}} \rho} \norm{u}_{L^2(B)}^2
		\ge c_F \norm{u}_{L^2(B)}^2 .
	\end{equation}

	For all $\rho \in \Bary_{\lambda, B}(\mathcal{Q})$ and $\nu \in \mathcal{Q}$ it holds that $-(\Phi^\nu)'(\rho)$ are Hermitian and nonnegative. They are also uniformly bounded since all $\nu$ and $D^2 \varphi_\rho^\nu$ are uniformly bounded away from zero according to~\eqref{ass_reg_P} and Proposition~\ref{prop:regularity}: namely, Theorem~\ref{thm:DiffMAmapping} together with the Poincar{\'e} inequality and Theorem 6.27 \cite{lieberman2013oblique} yield that there is a constant $C_\Phi > 0 $ such that
	\begin{equation}\label{eq:diff_Phi_bound}
		\norm*{(\Phi^\nu)'(\rho) u}_{H^2(B)} \le C_\Phi \norm{u}_{L^2(B)}.
	\end{equation}
	In particular, the operators $G$ and all $G_n$ are a.s.\ well-defined, uniformly positive definite, and thus continuously invertible in $L^2_\diamond(B)$ with 
	\[
	\norm{G^{-1}}_{L^2_\diamond(B)} \le \norm{F'(\rhob)^{-1}}_{L^2_\diamond(B)} \le \frac{1}{c_F},\quad \norm{G_n^{-1}}_{L^2_\diamond(B)} \le \frac{1}{c_F}.
	\]

	Now we consider the space $H^2_\diamond(B)$. 
	It is easy to see that $F'(\rho)$ and $(\Phi^\nu)'(\rho)$ can be continuously extended to it for any $\rho \in \Bary_{\lambda, B}(\mathcal{Q})$, $\nu \in \mathcal{Q}$. 
	We are going to show that there exist $G^{-1}$ and $G_n^{-1}$ for all $n$, and they are uniformly bounded.
	Obviously, $F'(\rhob)$ and $\int_0^1 F'(\rhob_n^t) \d t$ are continuously invertible, with uniformly bounded inverses. In particular, they are Fredholm operators of index $0$.
	Due to \eqref{eq:diff_Phi_bound} and the Rellich--Kondrachov theorem, $(\Phi^\nu)'(\rho)$ are compact  and uniformly bounded in $H^2_\diamond(B)$ for all $\rho \in \Bary_{\lambda, B}(\mathcal{Q})$, $\nu \in \mathcal{Q}$, as well as any of their average. Thus $G \eqset F'(\rhob) - \E (\Phi^\nu)'(\rhob)$ is a Fredholm operator, and $\operatorname{ind} G = \operatorname{ind} F'(\rhob) = 0$; since $G$ is positive definite in $L^2_\diamond(B)$, $\ker G = \{0\}$, therefore $G$ is invertible in $H^2_\diamond(B)$. The same applies for any $G_n$.
	Let us prove that $G_n^{-1}$ are uniformly bounded in $H^2_\diamond(B)$. Suppose $G_n u = v \in H^2_\diamond(B)$. Then
	\begin{align*}
		\norm*{\left(\int_0^1 F'(\rhob_n^t) \d t\right) u}_{H^2(B)} 
		&\le \norm{v}_{H^2(B)} + \frac{1}{n} \sum_{i = 1}^n \int_0^1 \norm*{(\Phi^{\nu_i})'(\rhob_n^t) u}_{H^2(B)} \d t \\
		&\le \norm{v}_{H^2(B)} + C_\Phi \norm{u}_{L^2(B)} \\
		&\le \norm{v}_{H^2(B)} + C_\Phi \norm{G_n^{-1}}_{L^2_\diamond(B)} \norm{v}_{L^2(B)} \\
		&\le \left(1 + \frac{C_\Phi}{c_F}\right) \norm{v}_{H^2(B)} .
	\end{align*}
	On the other hand,
	\[
	\norm*{\left(\int_0^1 F'(\rhob_n^t) \d t\right) u}_{H^2(B)} 
	\ge \norm*{\left(\int_0^1 F'(\rhob_n^t) \d t\right)^{-1}}_{H^2_\diamond(B)}^{-1} \norm{u}_{H^2(B)}
	\ge c \norm{u}_{H^2(B)} .
	\]
	Therefore, 
	\[
	\norm{G_n^{-1}}_{H^2_\diamond(B)} \le \frac{1}{c} \left(1 + \frac{C_\Phi}{c_F}\right) .
	\]

	Now let us prove that $G_n^{-1} \to G^{-1}$ in SOT. First, 
	\[
	\int_0^1 F'(\rhob_n^t) \d t \to F'(\rhob) \quad \text{a.s.}
	\]
	since $\rhob_n \xrightarrow{C^2(\bar{B})} \rhob$ a.s.\ by Theorem~\ref{thm:LLN}. Second, \eqref{eq:diff_Phi_bound}, the LLN, and separability of $H^2_\diamond(B)$ yield that 
	\[
	\frac{1}{n} \sum_{i = 1}^n (\Phi^{\nu_i})'(\rhob) \xrightarrow{\mathrm{SOT}} \E (\Phi^{\nu})'(\rhob) \quad \text{a.s.}
	\]
	It remains to show that 
	\begin{equation}\label{eq:SOT_conv}
		\frac{1}{n} \sum_{i = 1}^n \int_0^1 (\Phi^{\nu_i})'(\rhob_n^t) \d t - \frac{1}{n} \sum_{i = 1}^n (\Phi^{\nu_i})'(\rhob) \xrightarrow{\mathrm{SOT}} 0 \quad \text{a.s.}
	\end{equation}
	Let $\nu \in \mathcal{Q}$ and $\rho \xrightarrow{C^2(\bar{B})} \rhob$. Due to Theorem~\ref{thm:DiffMAmapping} 
	\[
	\norm*{(\Phi^\nu)'(\rho) u - (\Phi^\nu)'(\rhob) u}_{H^2(B)} \to 0 \text{~~for any~~} u \in C^{0, \alpha}_\diamond(\bar{B}) ,
	\]
	hence \eqref{eq:diff_Phi_bound} and the density of $C^{0, \alpha}_\diamond(\bar{B})$ in $L^2_\diamond(B)$ yield that $(\Phi^\nu)'(\rho) \to (\Phi^\nu)'(\rhob)$ in SOT on $H^2_\diamond(B)$.
	Now we fix $u \in H^2_\diamond(B)$, then the functions 
	\[
	f^{\nu_i}(\rho) \eqset \norm*{\int_0^1 (\Phi^{\nu_i})'(\rho^t) u \d t - (\Phi^{\nu_i})'(\rhob) u}_{H^2(B)},
	\]
	where $\rho^t \eqset (1 - t) \rhob + t \rho$, are bounded, continuous, and $f^{\nu_i}(\rhob) = 0$. Since $\rhob_n \xrightarrow{C^2(\bar{B})} \rhob$ a.s., Lemma~\ref{lem:function_LLN} ensures that 
	\[
		\norm*{\frac{1}{n} \sum_{i = 1}^n \int_0^1 (\Phi^{\nu_i})'(\rhob_n^t) u \diff t - \frac{1}{n} \sum_{i = 1}^n (\Phi^{\nu_i})'(\rhob) u}_{H^2(B)}
		\le \frac{1}{n} \sum_{i = 1}^n f^{\nu_i}(\rhob_n) 
		\to 0 \quad \text{a.s.}
	\]
	Taking a dense countable set $\{u_j\}_{j \in \N}$ in $H^2_\diamond(B)$ and using boundedness of $(\Phi^\nu)'$  by \eqref{eq:diff_Phi_bound} one obtains~\eqref{eq:SOT_conv}. 
	Combining the above results we conclude that $G_n \xrightarrow{\mathrm{SOT}} G$ a.s. 
	Finally, for any $u \in H^2_\diamond(B)$ one has
	\[
	G_n^{-1} u - G^{-1} u = G_n^{-1} (G - G_n) G^{-1} u \to 0
	\]
	since $G_n^{-1}$ are uniformly bounded. We thus have shown $G_n^{-1} \xrightarrow{\mathrm{SOT}} G^{-1}$ a.s.
	
	\textbf{Step 3.}
	Note that $\norm{\varphi_{\rhob}^\nu}_{H^2(B)} \le C \norm{\varphi_{\rhob}^\nu}_{C^2(\bar{B})}$, thus $\E \norm{\varphi_{\rhob}^\nu}_{H^2(B)}^2 < \infty$, and by the standard CLT in Hilbert spaces (see e.g.\ \cite[Theorem~10.5]{ledoux2013probability}) applied to $\{\varphi_i\}_{i \in \N}$ we obtain that
	\begin{equation}\label{eq:CLT_phi}
		\frac{S_n}{\sqrt{n}} \eqset \frac{1}{\sqrt{n}} \sum_{i = 1}^n \left(\varphi_i - \E \varphi_{\rhob}^\nu\right) \xrightarrow{d} \xi, \quad \xi \sim \mathcal{N}\left(0, \Var(\varphi_{\rhob}^\nu)\right).
	\end{equation}
	According to~\eqref{eq:char_linear},
	\[
	\sqrt{n} \left(\rhob_n - \rhob\right) = G_n^{-1} \frac{S_n}{\sqrt{n}}.
	\]
	Since $G_n^{-1}$ are uniformly bounded and $G_n^{-1} \xrightarrow{\mathrm{SOT}} G^{-1}$ a.s., Lemma~\ref{lem:Slutsky} yields the CLT for $\rhob_n$:
	\[
	\sqrt{n} \left(\rhob_n - \rhob\right) \xrightarrow{d} G^{-1} \xi \sim \mathcal{N}\left(0, G^{-1} \Var(\varphi_{\rhob}^\nu) G^{-1}\right).
	\]
\end{proof}

\begin{remark}[CLT in the discrete case] 
	Let us finally remark that the above delta-method can easily be adapted to the case where $\Omega$ is convex and bounded and $P$ is supported on a set of measures of the form $\sum_{j=1}^N \nu^j \delta_{x^j}$ with a lower bound on the mass of atoms $\nu^j \geq \eps$ and the distance between atoms $\norm{x^i - x^j} \ge \eps$ once $i \neq j$, for some $\eps > 0$. Of course, in this case, one cannot use the regularity theory for Monge--Amp\`ere and its linearization, but one can instead take advantage of the fine analysis of the semi-discrete case by Kitagawa, M\'erigot and Thibert \cite{kmt}. Indeed, in this case $\Phi^{\nu}(\rho)$ always take the form $x \mapsto \max\{x \cdot x^j - \psi^j\}$, but it follows from Theorem~5.1 in \cite{kmt} and the implicit function theorem that the dual variables $\psi^j$ depend in a $C^1$ way on $\rho \in C^{0,1}(\Omb)$ as well as an estimate of the form
	\[
	\norm*{(\Phi^\nu)'(\rho) u}_{L^2(\Omega)} \le C \sqrt{\sum_{j=1}^N u(V_j)^2}
	\]
	where the $V_j$'s are the Laguerre cells associated with $\rho$ and $\nu$. Note that the right hand side of this inequality depends on finitely many linear functionals of $u$, therefore $(\Phi^\nu)'(\rho)$ is compact in $L^2_\diamond(\Omega)$. Moreover, one can show that there is a uniform bound
	\[
	\norm*{(\Phi^\nu)'(\rho)}_{L^2_\diamond(\Omega)} \le C
	\]
	in the same way as in the proof of Theorem~5.1 in \cite{kmt}, using the Cheeger inequality for graphs together with the lower bounds on $\nu^j$ and $\norm{x^i - x^j}$, a relative isoperimetric inequality, and uniform bounds on $\rho$.
	Since $-(\Phi^{\nu})'(\rho)$ is Hermitian and nonnegative definite, we can argue as in the proof above invoking the Fredholm alternative theorem to invert the operators $G$ and $G_n$ in $L^2_\diamond(\Omega)$. This easily yields a CLT in $L^2_\diamond(\Omega)$ in this discrete setting. 
\end{remark}


\appendix
\section{Linearization of Monge--Amp\`ere equations}\label{sec-app}

We consider $\nu, \mu \in \left\{ \varrho \in \P_{\mathrm{ac}}(\bar B): \norm{\varrho}_{C^{0,\alpha}(\bar B)} + \norm{\log \varrho}_{L^\infty(\bar B)} < \infty \right\}$ on a closed ball $\bar B \eqset \overline{B_R(0)}$ of radius $R>0$ for $\alpha \in (0,1)$. Our goal is to linearize the following Monge--Amp\`ere equation with a second boundary value condition 
\begin{equation}\label{eq:MA_bounded}
\begin{array}{rcl}
\det(D^2 \varphi) \nu (\nabla \varphi) &=& \mu, \\
\nabla \varphi(\bar B) &=& \bar B,
\end{array}
\end{equation}
for some fixed $\nu \in C^{1,\alpha}(\bar B)$. Note that thanks to Brenier's theorem there exists a unique convex solution satisfying \eqref{eq:MA_bounded} (a priori in the sense of $\nabla \varphi_\# \mu = \nu$), and it is $C^{2,\alpha}(\bar B)$ thanks to regularity theory for Monge--Amp\`ere equations.
We will need the following lemmas.

\begin{lemma}\label{lem:diff_boundary_equiv}
	Let $\varphi \in C^1(\bar B)$ be strictly convex. Then the following are equivalent
	\begin{itemize}
	\item $\nabla \varphi(\bar B) = \bar B$,
	\item $\nabla \varphi(\partial B) \subset \partial B$.
	\end{itemize}
\end{lemma}

\begin{proof}
	For the first direction assume by contradiction that there is $p \in \partial B$ such that $\nabla \varphi(p) \in \mathring B$. Note that at $p$ there is an outer normal to $\bar B$, namely $p$ itself. Take $a > 0$ such that $\nabla \varphi(p) + a p \in B$. Since $\nabla \varphi$ is surjective, there is $q \in \bar{B}$ satisfying $\nabla \varphi(q) = \nabla \varphi(p) + a p$. Then
	\[
	(q - p) \cdot (\nabla \varphi(q) - \nabla \varphi(p)) = a (q - p) \cdot p < 0,
	\]
	which contradicts the monotonicity of $\nabla \varphi$.
	
	For the other direction, note that $\nabla \varphi (\partial B) \subset \partial B$ implies $\nabla \varphi (\partial B) = \partial B$ since the only subset of $\partial B$ homeomorphic to $\partial B$ is $\partial B$ itself. 
	Now, by a similar argumentation as above one can obtain that $\nabla \varphi (\partial B) \subset \partial \nabla \varphi (\bar B)$.
	Furthermore, strict convexity of $\varphi$ yields that its conjugate $\varphi^* \in C^1(\R^d)$, and recall that $(\nabla \varphi^*)^{-1}\left(\{x\}\right) = \left\{\nabla \varphi(x)\right\}$ for any $x \in \mathring{B}$. Thus $\nabla \varphi$ maps $\mathring B$ to the interior of $\nabla \varphi (\bar{B})$. Therefore, $\partial \nabla \varphi (\bar B) = \nabla \varphi (\partial B) = \partial B$.
	Now, there is only one compact set in $\R^d$ with nonempty interior and boundary $\partial B$: $\bar B$, so that we have $\nabla \varphi(\bar B) = \bar B$.
\end{proof}

\begin{remark}
	Note that the proof of Lemma \ref{lem:diff_boundary_equiv} extends to the setting described in Remark \ref{rem:GenRegBdry}. Keeping the same notation the statement changes to equivalence of
	\begin{itemize}
	\item $\nabla \varphi(\bar \Omega) = \supp \nu$,
	\item $\nabla \varphi(\partial \Omega) \subset \partial \supp \nu$.
	\end{itemize}
	The contradicting argument reads in this case (since $\nabla H(p)$ is the outer normal at $p$)
	\[
	(q - p) \cdot (\nabla \varphi(q) - \nabla \varphi(p)) = a (q - p) \cdot \nabla H (p) < 0,
	\]
	where the inequality is strict due to the strong convexity.
\end{remark}

\begin{lemma}\label{lem:gradient_oblique}
For $\varphi \in C^{2}(\bar B)$ strongly convex such that $\norm{\nabla \varphi(x)}^2 - R^2 = 0$ for $x \in \partial B$, there is $\beta \in C(\partial B), \beta > 0$ such that $(D^2 \varphi)^{-1}(x) \cdot x = \beta(x) \nabla \varphi(x)$ for $x \in \partial B$. Futhermore, there exists $\kappa > 0$ such that $|\nabla \varphi(x) \cdot x| \geq \kappa$ for all $x \in \partial B$.
\end{lemma}
\begin{proof}
Note that the Legendre transform $\varphi^*$ is at least $C^2(\bar B)$. Indeed, by standard regularity theory for convex functions $\varphi^* \in C^1(\bar B)$ and since $D^2\varphi$ is invertible, the inverse function theorem applied to $\nabla \varphi$ yields differentiability for $\nabla \varphi^* = \left( \nabla \varphi \right)^{-1}$.
Now note that $\nabla \varphi^*$ also satisfies (see Lemma~\ref{lem:diff_boundary_equiv})
\begin{align*}
	\norm{\nabla \varphi^* (y)}^2 - R^2 &\leq 0 \, \text{~ for~all~} y \in  B \\
	\norm{\nabla \varphi^* (y)}^2 - R^2 &= 0 \, \text{~ for~all~} y \in \partial B.
\end{align*}
This implies by differentiating at a boundary point $y$ that there is $\tilde\beta(y)\geq 0$ such that
\begin{align*}
D^2 \varphi^*(y) \nabla \varphi^*(y) = \tilde\beta(y)y.
\end{align*}
By invertibility of $D^2 \varphi^*(y)$, we see that $\tilde\beta(y)> 0.$
Substituting $\nabla \varphi(x) = y$ gives by using properties of Legendre transform
\begin{equation*}
\begin{array}{lrcl}
&D^2 \varphi^*(\nabla \varphi(x)) \nabla \varphi^*(\nabla \varphi(x)) &=& \tilde\beta(\nabla \varphi(x)) \nabla \varphi(x), \\
\Longleftrightarrow &(D^2 \varphi)^{-1}(x) x &=& \tilde\beta(\nabla \varphi(x))\nabla \varphi(x),
\end{array}
\end{equation*}
for all $x \in \partial B.$ Set $\beta(x) \eqset \tilde\beta(\nabla \varphi(x))$ and note that $\beta \in C(\partial B)$ since
\begin{equation*}
\beta(x) = \frac{1}{R^2} (D^2 \varphi)^{-1}(x)x \cdot \nabla \varphi(x).
\end{equation*}
The second statement follows since
\begin{align*}
|\nabla \varphi(x) \cdot x| = \frac{1}{\beta(x)} (D^2 \varphi)^{-1}(x)x \cdot x \geq R^2 K > 0,
\end{align*}
where $K$ is a constant only depending on $\norm{D^2 \varphi}_{L^\infty(B)}$ and $\norm{(D^2 \varphi)^{-1}}_{L^\infty(B)}.$
\end{proof}

\begin{remark}
One may readily check that Lemma \ref{lem:gradient_oblique} extends to the setting described in Remark \ref{rem:GenRegBdry}. Indeed, the proof uses that the defining convex functions of the ball are of the form $\norm{\cdot}^2 - R^2$ and the outer normal at a boundary point $x$ is $x$.
\end{remark}

Form now on, we fix the constant by considering potentials in the set 
\[C^{k,\alpha}_\diamond (\bar B) \eqset \left\{\varphi \in C^{k,\alpha}(\bar B): \int_B \varphi = 0 \right\}\]
with  $k \in \N$. Let us also define 
\[\mathcal{M} = \left\{ \varphi \in C^{2,\alpha}_\diamond(\bar B): \norm{\nabla \varphi}^2 - R^2 = 0 \text{~on~} \partial B\right\}.\]
We now prove that in a neighborhood of a strongly convex function $\varphi_0 \in \cM$ this set is the graph of a $C^1$-function.

\begin{lemma}\label{lem:M_loc_manif}
At $\varphi_0 \in \cM$ strongly convex, $\cM$ is locally given by the image of a bijective $C^1$-function on a closed subspace of $C_\diamond^{2,\alpha}(\bar B)$. More precisely, there exist open subsets $V \subset F_0 \eqset \left\{h \in C_\diamond^{2,\alpha}(\bar B): \nabla \varphi_0 \cdot \nabla h  = 0 \mathrm{~on~} \partial B \right\}$, $U \subset C_\diamond^{2,\alpha}(\bar B),$ with $\varphi_0 \in U,$ and a bijective $C^1$-function:
\begin{equation}
\chi_0 \colon V \rightarrow U \cap \cM.
\end{equation}
 
Furthermore, for $f_0 \eqset \Pi_{F_0}(\varphi_0),$ where $\Pi_{F_0}$ is the projection on $F_0$ defined by \eqref{eq:decomp_F}, it holds $\chi_0'(f_0) = \id.$
\end{lemma}

\begin{proof}
First, we show that $C^{2,\alpha}_\diamond(\bar B) = F_0 \oplus G_0$, where $F_0, G_0$ are linear subspaces defined as
\begin{align*}
F_0&\eqset \left\{ f \in C_\diamond^{2,\alpha}(\bar B): \nabla \varphi_0 \cdot \nabla f  = 0 \mathrm{~on~} \partial B\right\}, \\
G_0&\eqset \left\{ g \in C_\diamond^{2,\alpha}(\bar B): \exists c \in \R, -\div (A_0 \nabla g) = c \right\},
\end{align*}
where $A_0 = \cof(D^2 \varphi_0)$ is the cofactor matrix  of $D^2 \varphi_0.$ Here and in the following, $-\div (A_0 \nabla g) = c$ is to be understood in the distributional sense, i.e. for all $\psi \in C^\infty_c(B)$ such that $\int_B \psi = 0$
\[
\int_B A_0 \nabla g \cdot \nabla \psi = 0.
\]
Take $\varphi \in C_\diamond^{2, \alpha}(\bar B)$. 
Define $f$ to be a solution of
\begin{equation}\label{eq:decomp_F}
\begin{array}{rcll}
- \div (A_0 \nabla f) &=& - \div (A_0 \nabla \varphi) + \fint_B \div (A_0 \nabla \varphi) &\mathrm{~in~} B,\\
\nabla \varphi_0 \cdot \nabla f &=& 0 &\mathrm{~on~} \partial B,
\end{array}
\end{equation}
and $g$ a solution of
\begin{equation}\label{eq:decomp_G}
\begin{array}{rcll}
- \div (A_0 \nabla g) &=& -  \fint_B \div (A_0 \nabla \varphi) &\mathrm{~in~} B,\\
\nabla \varphi_0 \cdot \nabla g &=& \nabla \varphi_0 \cdot \nabla \varphi  &\mathrm{~on~} \partial B.
\end{array}
\end{equation}
Thanks to Lemma \ref{lem:gradient_oblique} the boundary conditions are uniformly oblique and compatible with the right hand side. Hence, both \eqref{eq:decomp_F} and \eqref{eq:decomp_G} admit a unique weak solution $f,g \in H^1_\diamond (B)$ which is $C_\diamond^{2,\alpha}(\bar B)$ thanks to linear elliptic PDE theory (e.g. by a combination of \cite[Theorem 5.54]{lieberman2013oblique} and \cite[Theorem 6.31]{gilbarg2015elliptic}). 
Thus, we have found a decomposition $\varphi = f + g$ for $f \in F_0$ and $g \in G_0.$ It is also unique because $F_0 \cap G_0 = \{0\}$. To see that notice that every $h \in F_0 \cap G_0$ satisfies
\begin{equation}\nonumber
\begin{array}{rcll}
- \div (A_0 \nabla h) &=& c &\mathrm{~in~} B, \mathrm{~ c \in~} \R\\
\nabla \varphi_0 \cdot \nabla h &=& 0  &\mathrm{~on~} \partial B,
\end{array}
\end{equation} whose unique solution is $h=0$.
In total, we obtain well-definedness of the projection operators $\Pi_{F_0}\colon C^{2,\alpha}(\bar B) \rightarrow F_0$ and $\Pi_{G_0}\colon C^{2,\alpha}(\bar B) \rightarrow G_0$.
Continuity of $\Pi_{F_0}$ and $\Pi_{G_0}$ in $C^{2,\alpha}(\bar B)$ follows by the open mapping theorem, see e.g.~\cite[Theorem~2.10]{brezis2010functional}.

Now we would like to apply the implicit function theorem to
\begin{align*}
\Gamma\colon F_0 \oplus G_0 &\rightarrow C^{1, \alpha}(\partial B), \\
(f,g) &\mapsto \norm{\nabla (f+g)}^2 - R^2.
\end{align*} 
Its partial derivative at $\varphi_0$ with respect to $g_0 \eqset \Pi_{G_0}(\varphi_0)$ is given by
\begin{align*}
\frac{\partial}{\partial g} \Gamma(\varphi_0)h = 2 \nabla \varphi_0 \cdot \nabla h \mathrm{~on~} \partial B, \quad h \in G_0.
\end{align*}
Bijectivity of the derivative means existence and uniqueness of $h \in C_\diamond^{2,\alpha}(\bar B)$ such that
\begin{equation}
\begin{array}{rcl}
-\div (A_0 \nabla h) &=& c \mathrm{~on~} B, \mathrm{~ c \in~} \R,\\
\nabla \varphi_0 \cdot \nabla h &=& w \mathrm{~on~} \partial B,
\end{array}
\end{equation}
where $w \in C^{1,\alpha}(\partial B)$. By the same argumentation as above, this is the case if and only if $c = - \frac{1}{|B|} \int_{\partial B} w \det(D^2 \varphi_0) \beta_0$ where $\beta_0$ is as in Lemma \ref{lem:gradient_oblique}. Continuity follows again by the open mapping theorem.
Thanks to the implicit function theorem there are $U_F \subset F_0, U_G \subset G_0$ open ($\Pi_F(\varphi_0) \in U_F$, resp. $\Pi_G(\varphi_0) \in U_G$) such that $\tilde \chi_0\colon U_F \rightarrow U_G$ is $C^1$ and
\begin{align*}
\Gamma(f,g) = 0 \mathrm{~ for ~} (f,g) \in U_F \oplus U_G \Longleftrightarrow g = \tilde \chi_0(f).
\end{align*}
This implies that $\chi_0\colon U_F \rightarrow \cM \cap U_F \oplus U_G, f \mapsto f + \tilde \chi_0(f)$ is well-defined, $C^1$ and bijective. 

Finally, note that for $f_0 \eqset \Pi_{F_0}(\varphi_0),$ $h \in F_0$
\begin{align*}
0 = \frac{\d}{\d f} \Gamma(f_0, \tilde \chi_0(f_0))h = 2 \underbrace{\nabla \varphi_0 \cdot \nabla h}_{=0} + \frac{\partial}{\partial g}\Gamma(f_0, \tilde \chi_0(f_0)) \tilde \chi_0'(f_0)h.
\end{align*}
By invertibility of $\frac{\partial}{\partial g}\Gamma(f_0, \chi_0(f_0))$, we conclude $\tilde \chi_0'(f_0) = 0,$ hence $\chi_0'(f_0) = \id.$
\end{proof}
Now for $\varphi_0 \in \cM$ take $U \subset C_\diamond^{2,\alpha}(\bar B)$ from Lemma \ref{lem:M_loc_manif} (and possibly restrict it further such that any $\varphi \in U \cap \cM$ is strongly convex) and consider the map
\begin{equation}
\begin{aligned}
M_\nu \colon &U \cap \cM \rightarrow \left\{ u \in C^{0,\alpha}(\bar B): \int_B u = 1 \right\} \\
&\varphi \mapsto \det(D^2 \varphi)\nu(\nabla \varphi)
\end{aligned}
\end{equation}
where $\nu$ is a fixed probability density in the set $\mathcal{Q}$ defined in \eqref{ass_reg_P} with $k=1$. Note that this map is well-defined by Lemma~\ref{lem:diff_boundary_equiv}
and the fact that the push forward preserves the mass.
We want to ``take the derivative at $\varphi \in U \cap \cM$'' by pulling back $M_{\nu}$ to the linear space $F_0$ with the map $\chi_0$ from Lemma \ref{lem:M_loc_manif}.
\begin{proposition}\label{thm:lin_MA}
In the setting of Lemma \ref{lem:M_loc_manif}, let $\varphi \in U \cap \cM$ be strongly convex. Then $N_\nu \eqset M_\nu \circ \chi_0$ is continuously differentiable at $f \eqset \Pi_{F_0}\varphi$ and the derivative is given by
\begin{equation}
\begin{aligned}
N_\nu'(f) \colon F_0 &\rightarrow C_\diamond^{0,\alpha}(\bar B) \\
h &\mapsto \tr( A_\nu D^2 (\chi_0'(f)h)) +\det(D^2 \varphi)\nabla\nu(\nabla \varphi)\cdot\nabla (\chi_0'(f)h),
\end{aligned}
\end{equation}
where $F_0 = \left\{h \in C_\diamond^{2,\alpha}(\bar B): \nabla \varphi_0 \cdot \nabla h  = 0 \mathrm{~on~} \partial B \right\}$ \\
 and $A_\nu \eqset \nu(\nabla \varphi) \cof(D^2\varphi)$. \\
In addition, in the weak sense we have
$$N_\nu'(f)h = \div(A_\nu \nabla (\chi_0'(f)h)).$$
\end{proposition}

\begin{proof}
Let $h \in F_0$, then the directional derivative is given by
\begin{align*}
\frac{\diff}{\diff t}N_\nu(f+th)_{|t = 0}= &\frac{\diff}{\diff t}\det(D^2 \chi_0(f + th))\nu(\nabla \chi_0(f + th))_{|t = 0} \\
= &~ \tr( A_\nu D^2 (\chi_0'(f)h)) +\det(D^2 \varphi)\nabla\nu(\nabla \varphi)\cdot\nabla (\chi_0'(f)h).
\end{align*}
By continuity of $\chi'_0$, we can conclude that $N_\nu$ is continuously differentiable.
Now note that if $\varphi \in C^{3,\alpha}(\bar B)$
\begin{align*}
\div (A_\nu \nabla (\chi_0'(f)h)) 
=&~\div (\nu(\nabla \varphi)\cof(D^2 \varphi)\nabla (\chi_0'(f)h)) \\
=&~D^2 \varphi\nabla \nu(\nabla \varphi) \cdot \cof(D^2 \varphi)\nabla (\chi_0'(f)h)\\
& + \nu(\nabla \varphi) \div(\cof(D^2 \varphi)\nabla (\chi_0'(f)h)) \\
=&~ \det(D^2 \varphi) \nabla \nu(\nabla \varphi) \cdot \nabla (\chi_0'(f)h)\\
& + \nu(\nabla \varphi) \tr \left( \cof(D^2 \varphi) D^2 (\chi_0'(f)h)\right),
\end{align*}
where, in the last line, we have used that $\cof(D^2 \varphi)$ has divergence-free columns (see Lemma p.462 in \cite{evans2010partial}).
This yields
\begin{align*}
N_\nu'(f)h 
= \div (A_\nu \nabla (\chi_0'(f)h)).
\end{align*}
The same result follows for $\varphi \in C^{2,\alpha}(\bar B)$ (in the weak sense) by density.
\end{proof}

For fixed $\nu\in \mathcal{Q}$, consider now 
\[\mathcal{S} = \left\{ \varrho \in \P_{\mathrm{ac}}(\bar B): \norm{\varrho}_{C^{0,\alpha}(\bar B)} + \norm{\log \varrho}_{L^\infty(\bar B)} < \infty \right\}\]
and the map
\begin{equation}\label{eq_def:measToPotential}
\begin{array}{rll}
\Phi^\nu\colon \mathcal{S} &\rightarrow& \cM, \\
\mu &\rightarrow& \varphi, \text{ where } \varphi \text{ strongly convex and }\nabla \varphi_{\#}\mu = \nu.
\end{array}
\end{equation}
Note that this is well defined thanks to Brenier's theorem (Theorem~2.12 (ii) \cite{Villani2}) and regularity theory for Monge--Amp\`ere equations (Theorem~3.3 \cite{defig2014}). Furthermore, by the considerations before we can now prove that it is continuously differentiable.

\begin{thm}\label{thm:DiffMAmapping}
$\Phi^\nu$ as defined in \eqref{eq_def:measToPotential} is continuously differentiable.  More precisely, for every $\mu \in \mathcal{S}$, the value of $(\Phi^\nu)'(\mu)f$ at $f \in C_\diamond^{0,\alpha}(\bar B)$ is the unique solution $h \in C_\diamond^{2,\alpha}(\bar B)$ of the linearized equation
\begin{equation}\nonumber
\begin{aligned}
\div(A_\nu \nabla h) &= f \, \mathrm{~in~} \, B,\\
\nabla \varphi_0 \cdot \nabla h &= 0 \, \mathrm{~on~} \, \partial B,
\end{aligned}
\end{equation}
where $\varphi_0 = \Phi^\nu(\mu)$ and $A_\nu = \nu(\nabla \varphi_0) \cof(D^2 \varphi_0).$
\end{thm}
\begin{proof}
For $\varphi_0 = \Phi^\nu(\mu) \in \cM$ the derivative of $N_\nu$ at $f_0 = \Pi_{F_0}(\varphi_0)$ is given by Proposition~\ref{thm:lin_MA}. Invertibility of $N'_\nu(f_0)$ is equivalent to finding, for every $f \in C_\diamond^{0,\alpha}(\bar B)$, a unique 
\[
h \in F_0 = \left\{u \in C_\diamond^{2,\alpha}(\bar B): \nabla \varphi_0 \cdot \nabla u  = 0 \mathrm{~on~} \partial B \right\},
\]
such that
\begin{equation}\label{eq_linMA}
\begin{aligned}
\div(A_\nu \nabla h) &= f \, \mathrm{~in~} \, B,
\end{aligned}
\end{equation}
in the weak sense where we have used that $\chi_0'(f_0) = \id$ by Lemma \ref{lem:M_loc_manif}.
As before, by strong convexity of $\varphi_0$ equation \eqref{eq_linMA} is uniformly elliptic and the boundary conditions are compatible, so that by elliptic regularity theory there is a unique solution $h \in C_\diamond^{2,\alpha}(\bar B)$, satisfying the boundary condition $\nabla \varphi_0 \cdot \nabla h = 0$ on $\partial B$. 
With the inverse function theorem we conclude that there is an open (in $\mathcal{S}$) neighborhood $U$ containing $\mu$ such that $N_\nu^{-1}|_U\colon U \rightarrow N_\nu^{-1}(U)$ is a $C^1$-diffeomorphism.
By possibly further restricting $U$ (such that $N_\nu^{-1}(U) \subset V$ from Lemma \ref{lem:M_loc_manif}), we see that $\Phi^\nu|_U = \chi_0 \circ N_\nu^{-1}|_U $ is also $C^1$ in a neighborhood of $\mu$. We employ again $\chi_0'(f_0) = \id$ to conclude.

\end{proof}
\section{Auxiliary probability results}
\label{sec-app_prob}

\begin{lemma}\label{lem:function_LLN}
	Consider space $C_b(\X)$ of bounded continuous functions on a separable metric space $\X$ endowed with the topology of pointwise convergence. 
	Let $f_1, f_2, \dots$ be i.i.d.\ (Borel) random functions from $C_b(\X)$ s.t.\ $f_1(x^*) = 0$ a.s.\ and $\E \sup_{x \in \X} \abs{f_1(x)} < \infty$. 
	Let $\{X_n\}_{n \in \N}$ be a sequence of r.v.\ convergent to $x^*$ a.s. Then
	\[
	\frac{1}{n} \sum_{i = 1}^n f_i(X_n) \to 0 \quad \text{a.s.}
	\] 
\end{lemma}

\begin{proof}
	Consider the modulus of continuity for $f$ at point $x^*$:
	\[
	\omega_f(\delta, x^*) \eqset \begin{cases}
		\sup_{x \in B_\delta(x^*)} \abs{f(x) - f(x^*)}, & \delta > 0,\\
		0, & \delta = 0.
	\end{cases}
	\]
	Note that $(f, \delta) \mapsto \omega_f(\delta, x^*)$ is measurable: indeed, take a countable dense set $S \subset \X$, then
	\[
	\omega_f(\delta, x^*) = \sup_{x \in S} \abs{f(x) - f(x^*)} \ind[d(x, x^*) < \delta].
	\]
	Since $f_i(x^*) = 0$ a.s., we have for any fixed $\delta > 0$
	\begin{align*}
		\abs*{\frac{1}{n} \sum_{i = 1}^n f_i(X_n)}
		& \le \frac{1}{n} \sum_{i = 1}^n \omega_{f_i}(d(X_n, x^*), x^*) \\
		& \begin{aligned} \le \frac{1}{n} \sum_{i = 1}^n \bigl(&\omega_{f_i}(\delta, x^*) \ind\left[d(X_n, x^*) \le \delta\right] \\
		& + \sup_{x \in \X} \abs{f_i(X_n)} \ind\left[d(X_n, x^*) > \delta\right]\bigr).
		\end{aligned}
	\end{align*}
	Further, $\E \sup_{x \in \X} \abs{f_1(x)} < \infty$, therefore by the strong LLN 
	\begin{align*}
		\frac{1}{n} \sum_{i = 1}^n \sup_{x \in \X} \abs{f_i(X_n)} &\xrightarrow{a.s.} \E \sup_{x \in \X} \abs{f_1(x)}, \\
		\frac{1}{n} \sum_{i = 1}^n \omega_{f_i}(\delta, x^*) &\xrightarrow{a.s.} \E \omega_{f_1}(\delta, x^*) \le \E \sup_{x \in \X} \abs{f_1(x)}.
	\end{align*}
	Since $\ind\left[d(X_n, x^*) > \delta\right] \to 0$ a.s.\ it holds a.s.\ that
	\[
	\limsup \frac{1}{n} \sum_{i = 1}^n \omega_{f_i}(d(X_n, x^*), x^*)
	\le \E \omega_{f_1}(\delta, x^*) \to 0 \quad \text{as} \quad \delta \to 0
	\]
	due to Lebesgue's dominated convergence theorem. The claim follows.
\end{proof}

The following result is a version of Slutsky's theorem for Hilbert space.
We say that $X_n \in H$ converge in probability to $X$ ($X_n \xrightarrow{P} X$), if $\norm{X_n - X} \xrightarrow{P} 0$, i.e.\ for any $\eps > 0$ it holds that $P\left\{\norm{X_n - X} > \eps\right\} \to 0$.

\begin{lemma}\label{lem:Slutsky}
	Let $\{A_n\}_{n \in \N}$ be a sequence of random bounded operators on a separable Hilbert space $H$ convergent to a fixed operator $A$ in SOT a.s.\ and bounded in probability (i.e.\ for any $\eps > 0$ there exists $M_\eps$ s.t.\ $P\left(\norm{A_n} > M_\eps\right) \le \eps$ for all $n$).
	Let $\{X_n\}$ be a sequence of r.v.\ in $H$, $X_n \xrightarrow{d} X$.
	Then $A_n X_n \xrightarrow{d} A X$.
\end{lemma}

\begin{proof}
	Let $\{e_n\}_{n \in \N}$ be an o.n.b.\ in $H$ and $\Pi_k$ be the orthogonal projector onto the first $k$ axes $e_1, \dots, e_k$.
	Then
	\begin{equation}\label{eq:decomposition}
		A_n X_n = A X_n + (A_n - A) \Pi_k X_n + (A_n - A) \left(I - \Pi_k\right) X_n.
	\end{equation}
	Since $A_n \xrightarrow{\mathrm{SOT}} A$ a.s., for any fixed $k$ we have $\norm*{\left(A_n - A\right) \Pi_k}_{op} \to 0$ a.s., thus 
	\[
	\left(A_n - A\right) \Pi_k X_n \xrightarrow{P} 0.
	\]
	Moreover,
	\[
	\left(I - \Pi_k\right) X_n 
	\xrightarrow[n \to \infty]{d} \left(I - \Pi_k\right) X 
	\xrightarrow[k \to \infty]{P} 0.
	\]
	Since $A_n$ are bounded in probability, the above equations imply that
	\[
	(A_n - A) X_n \xrightarrow{P} 0.
	\]
	This together with~\eqref{eq:decomposition} and $X_n \xrightarrow{d} X$ yields convergence $A_n X_n \xrightarrow{d} A X$.
\end{proof}

\thanks{\textbf{Acknowledgments:} GC is grateful to the Agence Nationale de la Recherche for its support through the projects   MAGA (ANR-16-CE40-0014) and MFG (ANR-16-CE40-0015-01). KE acknowledges that this project has received funding from the European Union's Horizon 2020 research and innovation programme under the Marie Sk\l odowska-Curie grant agreement No 754362.  
The work of Alexey Kroshnin is supported by Russian Science Foundation grant No. 18-71-10108. The authors wish to thank Filippo Santambrogio who suggested an improvement in the statement of proposition \ref{logconcaveprop} and an anonymous referee who drew our attention to the discrete case.

\includegraphics[width=1.5cm]{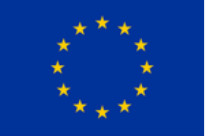}}

\bibliographystyle{plain}

\bibliography{bibli}

\end{document}